\documentclass[12pt,a4paper,english,bibliography=totoc]{scrartcl}

\usepackage{lmodern}
\usepackage{babel,csquotes,xpatch}
\usepackage[
    backend=biber,
    style=alphabetic,
    backref=true
]{biblatex}
\usepackage{enumitem}
\setenumerate{label=\textup{(\roman*)}}

\usepackage{amsmath, mathtools}
\usepackage{amssymb}
\usepackage{amsthm, thmtools}
\usepackage{tikz-cd}
\usepackage{xcolor}
\usepackage{braket}
\usepackage{bm, bbm}
\usepackage{enumitem}

\usepackage{microtype}

\usepackage{etoolbox}

\usepackage[colorlinks]{hyperref}
\usepackage[nameinlink]{cleveref}

\makeatletter
\renewcommand*\author[1]{%
  \stepcounter{author}%
  \ifnum\c@author=1
    \gdef\@author{#1}%
  \else
    \xdef\@author{\unexpanded\expandafter{\@author\and#1}}%
  \fi
  \csgdef{author@\the\c@author}{#1}}
\newcommand*\email[1]{%
  \csgdef{email@\the\c@author}{#1}}
\newcommand*\address[1]{%
  \csgdef{address@\the\c@author}{#1}}
\AtEndDocument{%
  \xdef\author@count{\the\c@author}%
  \c@author=1
  \print@authors}
\newcommand*\print@authors{%
  \ifnum\c@author>\author@count
  \else
    \print@author{\the\c@author}%
    \advance\c@author by 1
    \expandafter\print@authors
  \fi}
\newcommand*\print@author[1]{%
  \par\medskip
  \begin{tabular}{@{}l@{}}%
    \textsc{Addresses of \csuse{author@#1}}\\
    \csuse{address@#1}\\
    \textit{E-mail address}:
    \href{mailto:\csuse{email@#1}}{\texttt{\csuse{email@#1}}}
  \end{tabular}}
\makeatother

\declaretheorem[numberwithin=section]{theorem}
\declaretheorem[numberlike=theorem]{lemma, corollary,proposition}
\declaretheorem[numberlike=theorem,style=definition]{definition}
\declaretheorem[numberlike=theorem,style=remark,qed=$\diamondsuit$]{example}
\declaretheorem[numberlike=theorem,style=remark]{remark}

\declaretheoremstyle[notefont=\bfseries, notebraces={(}{)}]{conjecture}
\declaretheorem[numberlike=theorem,style=conjecture,refname={conjecture,conjectures},Refname={Conjecture,Conjectures}]{conjecture}

\declaretheorem[name=Theorem]{maintheorem}

\newcommand{\NN}{\mathbb{N}}
\newcommand{\ZZ}{\mathbb{Z}}
\newcommand{\QQ}{\mathbb{Q}}
\newcommand{\RR}{\mathbb{R}}
\newcommand{\CC}{\mathbb{C}}
\newcommand{\KK}{\mathbb{K}}
\newcommand{\PP}{\mathbb{P}}
\newcommand{\TT}{\mathbb{T}}

\DeclareMathOperator{\codim}{codim}
\DeclareMathOperator{\pos}{Pos}
\DeclareMathOperator{\weight}{wt}

\newcommand{\reg}{\mathrm{reg}}

\DeclareMathOperator{\rank}{rk}

\DeclareMathOperator{\Matroid}{M}

\DeclareMathOperator{\PGr}{\mathbb{G}}

\DeclareMathOperator{\trop}{trop}

\DeclareMathOperator{\Con}{Con}
\renewcommand{\Im}{\operatorname{Im}}
\DeclareMathOperator{\DD}{EDD}
\DeclareMathOperator{\DDgen}{EDD_{gen}}
\newcommand{\ChowGrp}{\mathrm{A}}

\title{Polar Degrees of Matroids}
\author{Clara Briand}
\address{École Normale Supérieure, France}
\email{clara.briand@ens.psl.eu}

\author{Leonie Kayser}
\address{Max Planck Institute for Mathematics in the Sciences, Germany}
\email{leo.kayser@mis.mpg.de}

\author{Julian Weigert}
\address{Max Planck Institute for Mathematics in the Sciences, Leipzig\\and Georg-August-Universität Göttingen, Germany}
\email{julian.weigert@mis.mpg.de}

\date{\today}

\begin{document}

\maketitle

\begin{abstract}
We show that the polar degrees of the coordinate-wise inverse of a linear subspace $L \subseteq \PP^n$ are given by the coefficients of a substitution of the reduced characteristic polynomial of the associated matroid $\Matroid(L)$. Our proof connects the geometry of conormal varieties of reciprocal linear spaces to the combinatorial conormal fan of $\Matroid(L)$. As a corollary, we settle two open conjectures regarding matroid discriminants.
\end{abstract}

\tableofcontents

\section{Introduction}

Let $L\subseteq \PP^n$ be a linear subspace of dimension $d$ not contained in any coordinate hyperplane and consider the rational \emph{Cremona map}
\[
\PP^n \dashrightarrow\PP^n, \qquad [x_0:\dots:x_n] \mapsto [x_0^{-1}:\dots:x_n^{-1}].
\]
The closure of the image of $L$ under this map is denoted $L^{-1}$ and is called the \emph{reciprocal linear space} associated to $L$. This a projective variety whose geometric properties are closely related to the realisable matroid $M$ represented by $L$. Reciprocal linear spaces have been widely studied both from the viewpoint of likelihood geometry in algebraic statistics and scattering equations \cite{HuhSturmfels2014,EFSS2021,BBT2025,KKT2025} and from the viewpoint of matroid theory \cite{orlik1992arrangements,Proudfoot2006,HuhKatz2012,Fink2020}. In this article we study $L^{-1}$ through the lens of projective duality. The central object in our work is the conormal variety $\Con(L^{-1})\subseteq \PP^n\times (\PP^n)^*$ whose general points are pairs $(p,H)$ where $p\in L^{-1}$ is a smooth point and $H\in (\PP^n)^*$ is a hyperplanes tangent to $L^{-1}$ at $p$. The projectively dual variety $(L^{-1})^\vee$ of $L^{-1}$ is the image of $\Con(L^{-1})$ under the projection to the second factor $(\PP^n)^*$.

Our initial motivation for studying these varieties stems from a collection of recent conjectures by Matsubara-Heo and Telen in their work on principal matroid determinants \cite{MatsubaraHeoTelen2026}. They introduce an analogue of the theory of principal $A$-determinants for toric varieties by Gelfand, Kapranov and Zelevinsky \cite{GKZ} to the  setting of reciprocal linear spaces, thereby defining the principal matroid determinant $E_L$. The key idea is that the stratification of toric varieties into orbits is replaced by the stratification of reciprocal linear spaces into strata corresponding to flats of the corresponding matroid $M\coloneq \Matroid(L)$. The main theorem \cite[Theorem 1.2]{MatsubaraHeoTelen2026} shows how the principal matroid determinant decomposes into a union of dual varieties of reciprocal linear spaces of minors of $M$. 
The following three conjectures about these factors and their multiplicities in $E_L$ were left open.

\begin{conjecture}[{\cite[Conjecture 7.2]{MatsubaraHeoTelen2026}}]
\label{con:dualdef}
The reciprocal linear space $L^{-1}$ is dual defective, i.e.\ $\dim\,(L^{-1})^\vee<n-1$,
if and only if the matroid $M$ is not connected.
\end{conjecture}
\begin{conjecture}[{\cite[Conjecture 7.3]{MatsubaraHeoTelen2026}}]
\label{con:dualdeg}
If the matroid $M$ is connected, then $(L^{-1})^\vee$ is a hypersurface of degree $2^d\beta(M)$, where $\beta(M)$ is the beta invariant of $M$.
\end{conjecture}
\begin{conjecture}[{\cite[Conjecture 7.1]{MatsubaraHeoTelen2026}}]
\label{con:multies}
    For each flat $F$ of $M$, such that $(L|_F^{-1})^\vee$ is a hypersurface, the defining polynomial of that hypersurface appears in the factorization of $ E_L$ with exponent equal to the local multiplicity of $L^{-2}$ along the stratum indexed by $F$. Here $L^{-2}$ denotes the image of $L^{-1}$ under the coordinate-wise squaring map.
\end{conjecture}

In this article we shall prove and strengthen the first two conjectures by describing the class of $\Con(L^{-1})$ in the Chow ring of $\PP^n\times (\PP^n)^*$, in other words, the sequence of \emph{polar degrees} of $L^{-1}$. Indeed, the dual degree is the last nonzero polar degree and we can read off dual defectiveness from the number of nonzero polar degrees. This motivates our first main theorem, which provides a combinatorial formula for these polar degrees.

To state it, let $\overline{\chi}_{M}(t)\in \ZZ[t]$ denote the reduced characteristic polynomial of the matroid $M$. We also write $\mu_i(L^{-1})$ for the $i$-th polar degree of $L^{-1}$, where the indexing is chosen such that $\mu_0(L^{-1})=\deg L^{-1}$. Using this notation we show the following.

\begin{maintheorem}\label{thm:intro:polardegrees}
Let $L\subseteq \PP^n$ be a linear subspace of dimension $0\leq d<n$ that is not contained in any coordinate hyperplane. Then
\[
\sum_{j = 0}^d \mu_j(L^{-1})t^j = (-2t-1)^d\, \overline{\chi}_M\left(\frac{2t}{2t+1}\right).
\]
Equivalently, setting $\delta_i(L^{-1}) \coloneq \mu_{d-i}(L^{-1})$ we obtain
\[
\sum_{i = 0}^d \delta_i(L^{-1})t^i = (-t-2)^{d} \, \overline{\chi}_M\left(\frac{2}{t+2}\right).
\]
\end{maintheorem}

\begin{corollary}\label{cor:introconjs}
    \Cref{con:dualdef,con:dualdeg} hold true.
\end{corollary}

From the viewpoint of the principal matroid determinant $E_L$, this describes which factors actually appear and what their degrees are in the factorization of $E_L$ due to \cite[Theorem 1.2]{MatsubaraHeoTelen2026}. The remaining \Cref{con:multies} concerns the multiplicities of these factors within $E_L$ and hence would finish the complete description of the factorization of $E_L$. The resolution of \Cref{con:multies} is the topic of our future work \cite{briand2025}.

Especially in recent years, expressing combinatorial invariants as geometric and tropical intersection numbers turned out to be a powerful tool for resolving longstanding conjectures in matroid theory \cite{HuhKatz2012,adiprasito2018hodge,braden2020singular,Ardila2022}. The theorem above can be seen as another work in this line of tradition as it expresses the coefficients of a substitution of $\overline{\chi}_M$ as intersection numbers on the conormal variety. Following this theme, \Cref{thm:intro:polardegrees} could be used to show the log-concavity of the coefficients of $(-2t-1)^d\overline{\chi}_M\left(\frac{2t}{2t+1}\right)$ for realizable matroids. We note however that this log-concavity follows more directly from the log-concavity of the coefficients of $\overline{\chi}_M$ which was proven in the groundbreaking work of Huh and Katz for realizable matroids \cite{HuhKatz2012} and Adiprasito, Huh and Katz for all matroids \cite{adiprasito2018hodge}.

Our proof of \Cref{thm:intro:polardegrees} uses a rational parametrization of $\Con(L^{-1})$ described in \cite{MatsubaraHeoTelen2026}, akin to the Horn--Kapranov uniformization.  We resolve the base locus of this map by passing from a product of projective spaces to the \emph{bipermutohedral variety}. The latter is a smooth projective toric variety whose corresponding fan hosts the \emph{conormal fan} of every matroid, introduced in \cite{Ardila2022}. We prove \Cref{thm:intro:polardegrees} in three steps that can be summarized as follows:

\begin{enumerate}[label=\arabic*.]
    \item Using tropical methods we pass from the rational parametrization of $\Con(L^{-1})$ to a similar, but slightly simpler map with a different image while tracking how this affects polar degrees (\Cref{sec:swimming}).
    \item We show that these new degrees can be computed as intersection numbers on the bipermutohedral variety that involve the class of the conormal fan of $M$. From this we derive a first combinatorial formula which sums over all decreasing flags of flats of $M$ (\Cref{sec:cycling}).
    \item Finally, we extract a recursion for the complicated combinatorial formula obtained in the previous step to relate the polar degrees to a substitution of the reduced characteristic polynomial $\overline{\chi}_{M}$ (\Cref{sec:running}).
\end{enumerate}

The structure of this article is as follows. \Cref{sec:warmup} serves as an introduction to the tools that we will use in proving \Cref{thm:intro:polardegrees}. We also provide direct proofs of \Cref{con:dualdef,con:dualdeg}. Especially the latter showcases our technique in a slightly simpler setting before we apply it to the conormal variety. The detailed proof of \Cref{thm:intro:polardegrees} is then presented in \Cref{sec:triathlon}, which is subdivided further into the three parts mentioned above. We finish the article providing several applications of \Cref{thm:intro:polardegrees} in \Cref{sec:applications}. This section contains examples, a discussion of Chern--Mather degrees of $L^{-1}$ and a generalization of \Cref{thm:intro:polardegrees} to arbitrary negative powers of linear spaces $L^{-m}$, $m\geq 1$.

\subsection*{Acknowledgements}
We would like to thank Simon Telen for proposing this project. C.B.\ thanks the Max Planck Institute for Mathematics in the Sciences for its hospitality and the Fondation de l'ENS for its financial support. J.W.\ was supported by the SPP 2458 “Combinatorial Synergies”, funded by the DFG grant 539677510.

\section{Preliminaries and dual degree} \label{sec:warmup}

Throughout this article, $L\subseteq \PP^n$ is a projective subspace of dimension $0\leq d<n$. We work over an algebraically closed field $\KK$ of characteristic 0 see however \Cref{exa:Fano} for a comment on this assumption. We always fix projective coordinates $\PP^n = \PP(\KK^E)$, $E=\{0,\dots,n\}$ and assume that $L$ is not contained in a coordinate subspace. Unless otherwise specified, $M$ is the matroid on the ground set $E$ associated to $L$.

\subsection{Matroids}

In this section we collect the required combinatorial preliminaries. Namely, we recall how to associate a matroid to a linear space, as well as some important properties of matroids. This construction is what allows us to study the geometric properties of reciprocal linear spaces from a combinatorial point of view.
For a thorough introduction to matroids, we refer to \cite{oxley2006matroid}.  

We start by recalling the definition of a matroid.

\begin{definition}\label{def: matroid}
A \emph{matroid} is a pair $M = (E,\mathcal{B})$ where $E$ is a finite set, called the ground set of $M$, and $\mathcal{B}$ is a nonempty family of subsets of $E$, called the \emph{bases} of $M$, such that the basis exchange axiom holds:
\[
\forall B_1,B_2\in \mathcal{B}:\forall e_1\in B_1\setminus B_2 : \exists e_2\in B_2\setminus B_1: (B_1\setminus e_1)\cup \{e_2\}\in \mathcal{B}.
\]
\end{definition}
We say that a subset $S\subseteq E$ is \emph{independent} if it is contained in a basis, otherwise we call $S$ \emph{dependent}. A \emph{circuit} of $M$ is an independent subset $S$ which is inclusion-minimal among all independent subsets. We further denote the \emph{rank} of any subset $S\subseteq E$ by $\rank_M(S)\coloneq \max_{B\in \mathcal{B}}|S\cap B|$. For $S=E$ we also simply write $\rank(M)\coloneq \rank_M(E)$, which is the cardinality of any basis of $M$. By $\text{cl}_M(S)\coloneq\set{x\in E| \rank_M(S\cup\{e\})=\rank_M(S)}$ we denote the \emph{closure} of $S$ in $M$. A subset $F\subseteq E$ is called a \emph{flat} of $M$ if $F=\text{cl}_M(F)$. The groundset $E$ is always a flat and $\text{cl}_M(\emptyset)$ is unique flat of rank zero. Its elements are called \emph{loops}. The beauty of matroid theory lies in the fact that most of the objects defined in this paragraph (bases, independent sets, circuits, rank function, closure, \dots) can be used to give \enquote{cryptomorphic definitions} of matroids.

Below we also work with \emph{flags of flats}, these are collections $\mathcal{F}$ of flats which are pairwise comparable under inclusion $\mathcal{F}= \{F_1\subseteq \dots \subseteq F_k\}$.
The size of such a collection is called the \emph{length} of $\mathcal{F}$ and denoted by $|\mathcal{F|}\coloneq k$. When working with loopless matroids we usually assume the flats in a flag to be non-empty.

As the terminology suggests, matroids are inspired by independence of vectors in linear algebra. We now explain how to associate a matroid $\Matroid(L)$ to a linear subspace $L\subseteq \PP(\KK^E)$, $E \coloneq \{0,\dots,n\}$.
Pick a matrix $A \in \KK^{(d+1)\times (n+1)}$ such that $d=\dim L$ and its projectivized row span is $L$, and label the columns of $A$ by the elements of $E$. These correspond to the $n+1$ coordinate linear functions restricted to $\widehat{L}$. We declare that a subset of $E$ is independent if and only if the column vectors (hence the linear forms) it indexes are linearly independent. In particular, the bases of $M$ correspond to non-vanishing maximal minors of $A$.
This construction is independent of the choice of the matrix $A$. Indeed, if $A'\in \KK^{(d+1)\times (n+1)}$ also has row span $L$, then there exists $B \in \operatorname{GL}_{d+1}(\KK)$ such that $A' = BA$. 


Operations on vector spaces also have matroid analogues. In this paper, we will mostly be concerned with the following operation.
Let $\widehat{L} \subset \KK^{n+1}$ denote the affine cone over $L$, and $L^\perp \coloneq \PP(\widehat{L}^\perp) \subseteq (\PP^n)^*$ the orthogonal complement of $L$. The matroid $M^\perp \coloneq \Matroid(L^\perp)$ is called the \emph{dual matroid} of $M$. More generally, for any matroid $M$, its dual $M^\perp$ is defined as a matroid with the same ground set as $M$, whose bases are the complements of the bases of $M$.

\begin{example}
\label{ex: matroidex}
Let $L_{\text{ex}}\subseteq \PP^6$ be the projectivized rowspan of the matrix $A_{\text{ex}}$ given by\begin{align*}
    A_{\text{ex}}  =
\begin{bmatrix}
0 & -1 & 1 & 0 & 0 & 1 & 0 \\
0 & 0 & -1 & 1 & 0 & 0 & 1 \\
1 & 0 & 0 & -1 & -1 & 0 & 0 \\
-1 & 1 & 0 & 0 & 0 & 0 & 0 \\
\end{bmatrix}.
\end{align*}
The matroid $M_{\text{ex}} = \Matroid(L_{\text{ex}})$ has ground set $E_{\text{ex}} = \{0,1,2,3,4,5,6\}$. It has 86 independent sets, 24 bases, 38 flats (illustrated in \Cref{fig:LatticeOfMex}), and 7 circuits (namely, $\{0, 1, 2, 3\}, \{0, 1, 4, 5\}, \{2, 3, 4, 5\},\{0, 1, 2, 4, 6\}, \{3, 4, 6\}, \{2, 5, 6\}, \{0, 1, 3, 5, 6\}$).
\end{example}
\begin{figure}
\centering
\includegraphics[width=\textwidth]{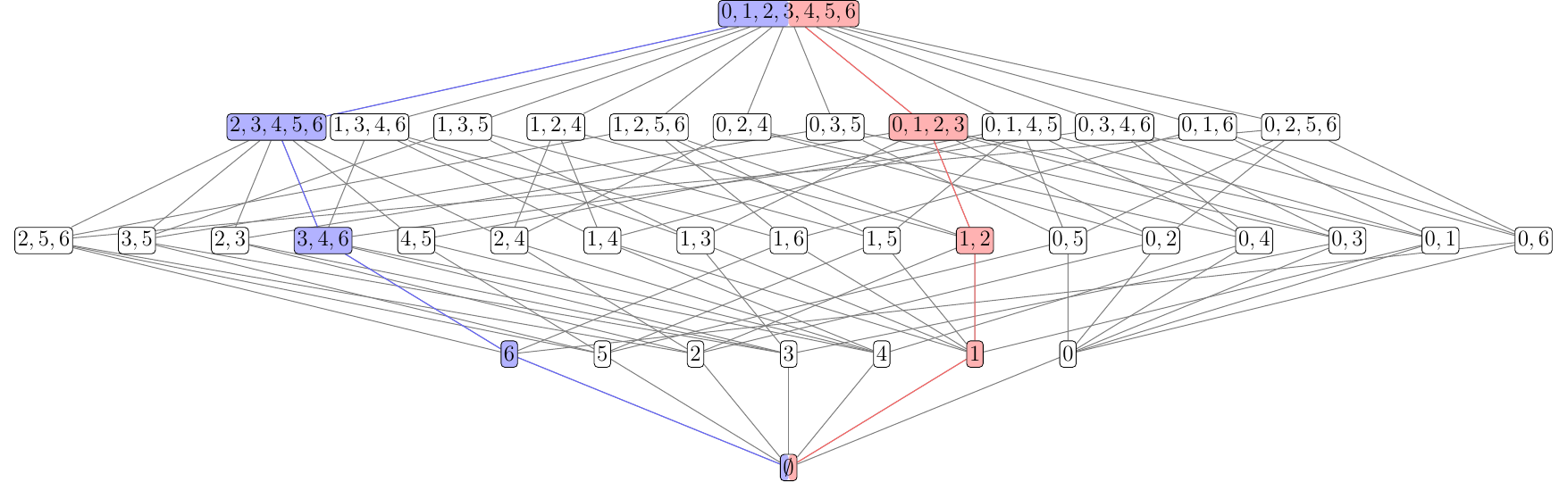}
\caption{Lattice of flats of $M_{\text{ex}}$. Highlighted are a flag of flats (\textcolor{red!60}{red}) and a \emph{decreasing} flag of flats (\textcolor{blue!60}{blue}), see \Cref{def:decreasing flag}.}
\label{fig:LatticeOfMex}
\end{figure}
We frequently use the (reduced) characteristic polynomial, the beta invariant and the unsigned Möbius invariant of the matroid $M$. For completeness we state the definition of these classical matroid invariants. For more details we refer to \cite{WhiteCombGeo1987}.

\begin{definition}
Let $M$ be a matroid of rank $d+1$ on $E$. The \emph{Möbius function} $\mu$ of $M$ assigns an integer to every pair of subsets $F \subseteq G \subseteq E$, $G$ a flat, as follows: If $F$ is \emph{not} a flat, then $\mu(F,G) = 0$, otherwise
\[
\mu(F,G)=\begin{cases}
   1 & \text{if } F=G, \\
   -\sum_{\substack{H\text{ flat}\\F \subseteq H\subsetneq G}}\mu(F,H) & \text{if }F \subsetneq G.
\end{cases}
\]
The \emph{characteristic polynomial} of $M$ is the univariate polynomial $\chi_M(t)$ defined as
\[
\chi_M(t) \coloneq \sum_{F\text{ flat}}\mu(\emptyset,F)t^{d+1-\rank_M(F)} = \sum_{S \subseteq E} (-1)^{|S|}t^{d+1-\rank_M(S)}.
\]
For $M\neq \emptyset$, this polynomial has a root at $t=1$, so we also define the \emph{reduced characteristic polynomial}
\[
\overline{\chi}_M(t)\coloneq\frac{\chi_M(t)}{t-1}.
\]
The \emph{(unsigned) Möbius invariant} and the \emph{beta invariant} of $M$ are defined as
\begin{align*}
\mu^+(M)&\coloneq(-1)^{d+1}\mu(\emptyset,E) = |\chi_M(0)|=(-1)^d\overline{\chi}_M(0) \geq 0 \\
\beta(M)&\coloneq (-1)^d \overline{\chi}_M(1) \geq 0.
\qedhere
\end{align*}
\end{definition}

\begin{example}
    \label{ex: matroidexcontinued}
    Let $M=M_{\text{ex}}$ be the matroid defined in \Cref{ex: matroidex}. We have \begin{align*}
        \chi_M(t)&=t^4-7t^3+19t^2-23t+10 \\
    \overline{\chi}_M(t)&=t^3-6t^2+13t-10\\
        \beta(M)&=2  \qquad \qquad \qquad        \mu^+(M)=10. \qedhere
    \end{align*}
\end{example}

The Möbius invariant vanishes if and only if $M$ has a loop, while the beta invariant vanishes if and only if $M$ is disconnected (or empty or a single loop). Here, a matroid is \emph{disconnected} if there exists a partition $E = E_1 \mathbin{\dot\cup} E_2$ such that $M = M_1 \oplus M_2$, where $M_i$ is a matroid on the ground set $E_i$ and the bases of $M$ are unions of bases from $M_1$ and $M_2$. Otherwise $M$ is called \emph{connected}. The set of such connected components of $M$ is well-defined, we denote the number of connected components of $M$ by $c(M)$.

In the case $M = \Matroid(L)$, $M$ is disconnected if and only if there exists a partition $E = E_1 \mathbin{\dot\cup} E_2$ as above such that
\[
\widehat{L} = (\widehat{L} \cap \KK^{E_1}) \oplus (\widehat{L}\cap \KK^{E_2}) \eqcolon \widehat{L}_1 \oplus \widehat{L}_2 \subseteq \KK^E,
\]
and in that case $\Matroid(L) = \Matroid(L_1) \oplus \Matroid(L_2)$.

It turns out that many of the algebro-geometric properties of $L^{-1}$ are determined by the matroid $M=\Matroid(L)$. For instance $\deg L^{-1} = \mu^+(M)$ (see for example the discussion in \Cref{sec:degreeexamples}), and $L^{-1}$ is a cone over a smaller reciprocal linear space if $M$ has a coloop.

\subsection{Polar degrees and conormal varieties}

We now switch gears and recall the fundamentals of projective duality, the most important of which are the definition of conormal variety and of polar degrees. For more details on projective duality we refer to \cite{GKZ,kleiman1986tangency},  for background on Chow groups and (multi)degrees see \cite{Fulton1998}.

We start by defining polar classes. Let $X \subset \PP^n$ be a projective variety of dimension $d$. Denote by $X_\reg$ the dense open set of regular points of $X$. Let $\Lambda \subseteq \PP^n$ be a linear subspace and consider the \emph{polar variety}
\[
P_\Lambda(X) \coloneq \overline{\Set{x\in X_\reg | \dim(T_xX \cap \Lambda) > \dim \Lambda - n + d}} \subseteq X.
\]
For general $(n-d+k-2)$-dimensional subspaces $\Lambda$, this algebraic subset is empty or equidimensional of the expected codimension $\codim_X P_\Lambda(X) = k$. Its class in the Chow group $\ChowGrp_{d-k}(X)$ is independent of the general choice of $\Lambda$.

\begin{definition}
The \emph{polar classes} of $X \subseteq \PP^n$ are the classes $P_k(X) = [P_\Lambda(X)] \in \ChowGrp_{\dim X-k}(X)$ for general $[\Lambda] \in \PGr(n-d+k-2,\PP^n)$, $k=0,\dots,d$. The \emph{polar degrees} of $X$ are the degrees $\mu_k(X) \coloneq\deg_{\PP^n} P_k(X)$.
\end{definition}

For example, $P_0(X) = [X] \in \ChowGrp_d(X)$, hence $\mu_0(X) = \deg X$. Our main result \Cref{thm:intro:polardegrees} concerns the polar degrees of the variety $L^{-1}$.
For more background on polar degrees and their place in projective geometry, see for example \cite{Piene1978,Piene1988}.

A useful approach to studying polar degrees of a projective variety $X \subseteq \PP^n$ is via its \emph{conormal variety}
\[
\Con(X) \coloneq \overline{\Set{(x,H) \in X_\reg \times (\PP^n)^* | H \supset T_{x}X }} \subset \PP^n \times (\PP^n)^*.
\]
This is an irreducible biprojective variety of dimension $n-1$.
The dual variety $X^\vee$ is the image of $\Con(X)$ under the projection $\operatorname{pr}_2 \colon \Con(X) \rightarrow (\PP^n)^*$.
It is the Zariski closure in $(\PP^n)^*$ of the set of hyperplanes of $\PP^n$ that are tangent to $X$ at a smooth point.

Consider the class of the conormal variety
\[
[\Con(X)] = 
\delta_0(X) h_1^{n}h_2 + \delta_1(X) h_1^{n-1}h_2^2 + \cdots + \delta_d(X) h_1^{n-d}h_2^{d+1} 
\in \ChowGrp^{n+1}(\PP^n \times (\PP^n)^*)
\]
in the Chow ring $\ChowGrp^\bullet(\PP^n \times (\PP^n)^*) = \ZZ[h_1,h_2]/\langle h_1^{n+1}, h_2^{n+1} \rangle$, where  $h_1 = [H \times (\PP^n)^*]$ and $h_2 = [\PP^n \times H']$, respectively.
In this notation, the bidegrees of the conormal variety encode the polar degrees as $\mu_k(X) =\delta_{d-k}(X)$ \cite[Section II.1]{kleiman1986tangency}. In other words, $\mu_{k}(X)$ counts the number of intersection points of $\Con(X)$ with a general pair of linear spaces  of type $[\PP^{n-d+k} \times \PP^{d-k+1}]$. From this, one can also deduce that $\mu_{d-n+1+\dim X^\vee}(X) = \deg X^\vee$.

For later reference, we also mention the behavour of polar degrees under the following construction: Let $X \subseteq \PP(\KK^a)$, $Y \subseteq \PP(\KK^b)$ be projective varieties with affine cones $\widehat{X}$, $\widehat{Y}$ (here we follow the convention $\widehat{\emptyset} \coloneq \{0\}$). The \emph{external join} of $X$ and $Y$ is the projective variety
\[
J(X,Y) \coloneq \PP(\widehat{X}\times \widehat{Y}) \subseteq \PP(\KK^{a+b}).
\]
The name comes from the fact that, under the natural inclusions $X,Y \hookrightarrow \PP^{a+b-1}$, $J(X,Y)$ can be interpreted as the union of lines connecting $X$ and $Y$ (assuming $X,Y\neq \emptyset$). This construction includes linearly degenerate subvarieties ($Y = \emptyset$) and cones over a projective subspace ($Y=\PP^{b-1}$).

In the following we denote the \enquote{bi-affine cone} of a bi-projective variety also by $\widehat{(-)}$.

\begin{lemma}\label{lem:convolution_of_polardegs}
Let $X \subseteq \PP(\KK^a)$, $Y \subseteq \PP(\KK^b)$ be projective varieties. The conormal variety of $J(X,Y)$ admits the following description of its affine cone (up to natural identifications)
\[
\widehat{\Con(J(X,Y))} = \widehat{\Con(X)} \times \widehat{\Con(Y)} \subseteq \KK^a \oplus (\KK^a)^* \oplus \KK^b \oplus (\KK^b)^*.
\]
Furthermore, $\mu_k(J(X,Y)) = \sum_{i+j=k} \mu_i(X)\mu_j(Y)$. Consequently, for nonempty $X,Y$ the dual variety $J(X,Y)^\vee$ is never a hypersurface.
\end{lemma}
\begin{proof}
The description of the conormal variety follows immediately from the observation that $T_{(\hat x,\hat y)}\widehat{X}\times \widehat{Y} = T_{\hat x} \widehat{X} \oplus T_{\hat y}\widehat{Y} \subseteq \KK^a \oplus \KK^b$.

The convolution formula for the polar degrees then follows, as the bihomogeneous coordinate ring of $\Con(J(X,Y))$ is the (graded) tensor product of the bihomogeneous coordinate rings of $\Con(X)$ and $\Con(Y)$. Therefore, the bigraded Hilbert series and hence also the multidegree of the varieties multiply.

Finally, the description of the conormal variety shows that $ J(X,Y)^\vee = J(X^\vee,\PP^{b-1}) \cap J(\PP^{a-1},Y^\vee)$ has codimension at least two.
\end{proof}

\begin{example}
If for our linear space $L \subseteq \PP(\KK^E)$ there exist a partition $E = E_1 \mathbin{\dot\cup} E_2$ with
\[
\widehat{L} = (\widehat{L} \cap \KK^{E_1}) \oplus (\widehat{L} \cap \KK^{E_2}) \eqcolon \widehat{L}_1 \oplus \widehat{L}_2,
\]
then not only $L = J(L_1,L_2)$, but more interestingly also $L^{-1} = J(L_1^{-1},L_2^{-1})$. In this way, a disconnected matroid $M$ corresponds to the reciprocal linear space splitting as an external join. \Cref{lem:convolution_of_polardegs} allows us in principle to reduce our discussion of \Cref{thm:intro:polardegrees} to the case of connected matroids, though most of our discussion will only rely on the condidition that $M$ has no coloops.
\end{example}



In the case of reciprocal linear spaces, the conormal variety admits a useful birational parametrization, similar to the Horn--Kapranov uniformization for toric varieties. Just as in the toric situation, the parametrization of $\Con(L^{-1})$ has an indeterminacy locus. We will therefore first have to resolve this locus before using this map to compute the polar degrees of $\Con(L^{-1})$.

Before defining the parametrization we introduce some notation that will be used throughout the paper.
Let $x = [x_0:\dots:x_n]$, $y = [y_0:\dots:y_n] \in \PP^n$, then $xy \coloneq [x_0y_0: \dots :x_ny_n]$ is the Hadamard product, and $x^m \coloneq [x_0^m:\dots:x_n^m]$ is the coordinate-wise $m$-th powers, $m \in \ZZ$, whenever these are defined. Also, we will denote the Hadamard product of varieties with $\star$, i.e.\ if $X,Y \subset \PP^n$ are two embedded projective varieties, then
\[
X \star Y \coloneq \overline{\Set{xy \in \PP^n | x\in X, y \in Y }} \subseteq \PP^n.
\]
\begin{proposition}\label{prop: parametrization}
The following rational map parametrizes $\Con(L^{-1})$
\[
\psi\colon  L \times L^\perp \dashrightarrow \PP \times (\PP^n)^* \qquad (x, y) \mapsto (x^{-1}, x^2y).
\]
\end{proposition}
Here and in the following, we denote by $\TT^n \coloneq (\KK^\times)^{n+1}/\KK^\times \subseteq \PP^n$ the algebraic torus.
\begin{proof}
Let $p \in L^{-1} \cap \TT^n \subseteq (L^{-1})_\reg$ and set $x \coloneq p^{-1}$.
By the proof of \cite[Proposition 4.1]{MatsubaraHeoTelen2026}, a hyperplane $H \subset \PP^n$ with coordinates $z \in (\PP^n)^*$ is tangent to $L^{-1}$ at $p$ if and only if there exists $y \in L^\perp$ such that $z = p^{-2} y = x^2y$. Since $p \in L^{-1}$, $x \in L$. Therefore, the image of the rational map $\psi$ contains the open subset
\[
\Set{(p, H) \in (L^{-1}\cap \TT^n) \times (\PP^n)^* | H \supset T_{p}(L^{-1}) } \subseteq \Con(L^{-1}).
\]
After closure, we have $\overline{\Im(\psi)} = \Con(L^{-1})$.
\end{proof}

By construction of $L^\perp$, $L\star L^{\perp}$ is always contained in the hyperplane $H_+ \coloneq V(x_0+\dots+x_n)$. As we will see, if $L^{-1}$ is not dual-defective, then this Hadamard product has dimension $n-1$ and therefore equals this hyperplane. This suggests that the map
\[
\phi \colon \PP^n \times (\PP^n)^*\dashrightarrow \PP^n \times (\PP^n)^*, \qquad (x,y) \mapsto (x^{-1}, xy) 
\]
is easier to understand than $\psi$. The main objective of the next section is to find the relationship between the image of $L\times L^\perp$ under $\psi$ and under the simpler map $\phi$. To do so, we use techniques from tropical geometry.

\subsection{Tropicalization and non-defectivity}

In this section we will recall \textit{tropical linear spaces}. This will allow us to give a combinatorial proof of \Cref{con:dualdef}. Also, we will be able to reduce the problem of determining the dual degree of $L^{-1}$ to the computation of the degree of the rational map 
\[
L\times L^{\perp} \dashrightarrow H_+= V(z_0+\dots+z_n) \subseteq (\mathbb{P}^n)^*, \qquad (x, y) \mapsto xy.
\]
We begin by recalling the fundamentals of tropical geometry and tropical linear spaces. For a thorough introduction, see \cite{Maclagan2015}.

There are several ways to define the tropicalization of an embedded projective variety. Here, we decide to go with the standard modern definition, which arises from monomial orderings. Let $e_E \in \RR^E$ be the vector whose entries are all equal to 1, recall that
every vector $\omega \in \RR^{n+1}/\RR e_E$ defines a partial monomial ordering on the variables $x_0,\dots,x_n$ of the coordinate ring of  $\TT^n$ by 
\[
x_0^{\alpha_0}\dotsm x_n^{\alpha_n} \leq x_0^{\beta_0}\dotsm x_n^{\beta_n} \iff \omega_0\alpha_0 + \dots + \omega_n\alpha_n \leq\omega_0\beta_0 + \dots + \omega_n\beta_n.
\]
The initial term of a polynomial with respect to $\omega$ is the set of monomials of highest $\omega$-weight appearing in its support.
For $I \subset \KK[\TT^n]$ an ideal,
the \emph{initial ideal of $I$ with respect to $\omega$} is
\[
\text{in}_\omega(I) \coloneq \langle \text{in}_\omega(f) \mid f \in I\rangle.
\]
Let $X^\circ \subset \TT^n$ be a $d$-dimensional irreducible subvariety of $\TT^n$. Its \emph{tropicalization} is 
\[
\trop(X^\circ) = \Set{\omega \in \RR^{n+1}/\RR e_E | \text{in}_\omega(X^\circ) \neq \langle 1 \rangle}.
\]
If $X \subset \PP^n$ is an embedded projective variety we set $\trop(X) \coloneq \trop(X \cap \TT^n)$.

The \emph{Structure Theorem} \cite[Theorem 3.3.5]{Maclagan2015} states that if $X^\circ$ is an irreducible $d$-dimensional subvariety of $\TT^n$, then $\trop(X^\circ)$ is the support of a \emph{weighted balanced rational polyhedral complex} $\mathcal{C}_{X^\circ}$ pure of dimension $d$, i.e.\ all maximal cells have dimension $d$. The adjectives \emph{weighted balanced} mean that to each maximal cell $\sigma$ of $\mathcal{C}_{X^\circ}$ we assign a non-negative integer called the \emph{weight} of $\sigma$, in such a way that it respects the \emph{balancing condition}. We will explain this construction for linear spaces below, for the general setting, see \cite[Section 3.4]{Maclagan2015}.
It is important to keep in mind that there are many polyhedral complexes with support $\trop(X)$. In what follows, unless stated otherwise, when we refer at $\trop(X)$ as a polyhedral complex, we mean the coarsest polyhedral complex with support $\trop(X)$. 

Tropicalization allows us to express Hadamard products and powers of varieties in a simple way.

\begin{lemma}[{\cite[Lemma 4.3]{MatsubaraHeoTelen2026}}]
\label{lem: tropLinSpaces}
Let $X,Y \subset \mathbb{P}^n$ be irreducible projective varieties such that $X = \overline{X^{\circ}}$ and $Y = \overline{Y^{\circ}}$. Then we have the following equalities of subsets of $\RR^{n+1}/\RR e_E$:
\begin{align*}
\trop(X \star Y) &= \trop(X) + \trop(Y) \\
\trop(X^m) &= \trop(X), \qquad m\geq 1,
\end{align*}
where $+$ denotes the Minkowski sum of sets. In general, these equalities do not hold at the level of weighted polyhedral complexes.
\end{lemma}

As an immediate consequence, using \Cref{prop: parametrization}, we obtain the following equality of sets:
\[
\trop((L^{-1})^\vee) = \trop(L) + \trop(L^\perp)
\]

When $L$ is a linear space, we call $\trop(L)$ a \textit{tropicalized linear space}. In this case
we can endow $\trop(L)$ with a fan structure with nice combinatorics, which is called the \textit{Bergman fan} $\Sigma_L$ of $L$. This can cause confusion: Indeed, some authors define the Bergman fan of a linear space $L$ as the coarsest fan structure on $\trop(L)$. Here, we decide to go with the convention in \cite{Ardila2022}, where the Bergman fan is a refinement of the coarse fan.

The construction of the Bergman fan of a matroid $M$ goes as follows.
We represent each flat $F$ of $M$ by the incidence vector $ e_F \coloneq \sum_{i \in F}  e_i \in \RR^{E}/\RR e_E $, where $ e_i$ is the $i$-th basis vector. For any flag of flats $\mathcal{F} = (\emptyset \subsetneq F_{1} \subsetneq \dots \subsetneq F_k \subsetneq E)$ of $M$ we consider the polyhedral cone positively spanned by the incidence vectors
\[
C_{\mathcal{F}} \coloneq \pos( e_{F_1}, \dots,  e_{F_k}) + \RR e_E \subset \RR^{E}/\RR e_E.
\]
This is a $k$-dimensional simplicial cone in $\RR^{E}/\RR e_E$.
By \cite[Theorem 4.2.6]{Maclagan2015}, the collection of cones $\set{ C_{\mathcal{F}} | \mathcal{F} \text{ flag of flats of }M}$ forms a pure simplicial fan of dimension $\rank(M)-1$ in $\RR^{E}/\RR e_E$. The support of this fan equals the tropicalized linear space $\trop(L)$.

Since this construction only uses the set of flats of $M$, it can be repeated for non-realizable matroids. We call the fan obtained in this way the \textit{Bergman fan} of $M$ and denote it $\Sigma_M$. Fans that can be realized as Bergman fans of some matroid are called \textit{tropical linear spaces}.

\begin{example}
The Bergman fan of $M_{\text{ex}}$ is a $3$-dimensional fan in $\RR^7/\RR e_E$. It has 36 rays which are labeled by the 36 flats $\emptyset \neq F \subsetneq E$ of the matroid. It has 78 maximal cones labeled by the 78 flags of flats of the form $\emptyset \subsetneq F_1 \subsetneq F_2 \subsetneq F_3 \subsetneq E$.
\end{example}

As a direct application of this tropical machinery, we prove \Cref{con:dualdef}, that is, we characterize when $\codim_{(\PP^n)^*} (L^{-1})^\vee > 1$.

\begin{theorem}\label{thm: defectiveness}
The reciprocal linear space $L^{-1}$ is dual-defective if and only if the matroid $M$ is not connected.
\end{theorem}

\begin{remark}\label{rmk: easydirectionofdualdef}
If $M$ is disconnected, then there exists a partition $E = E_1 \mathbin{\dot\cup} E_2$ such that $L^{-1} = J(L_1^{-1},L_2^{-1})$, hence $L^{-1}$ is dual-defective by \Cref{lem:convolution_of_polardegs}.
\end{remark}


For the reverse implication, recall that by \Cref{lem: tropLinSpaces}, set-theoretically
\[
\trop((L^{-1})^\vee) = \trop(M) + \trop(M^\perp) = \Sigma_M + \Sigma_{M^\perp},
\]
where $M^\perp$ denotes the dual matroid. In particular,
\begin{equation}\label{eq: tropDual}
\dim ((L^{-1})^\vee) = \dim(\Sigma_M + \Sigma_{M^\perp}) = \dim(L \star L^\perp).
\end{equation}


\begin{lemma}\label{lem: connected}
Suppose $M$ is a connected matroid. There exist flags of flats $\mathcal{F}$ in $M$ and $\mathcal{G}$ in $M^\perp$ such that the Minkowski sum $C_{\mathcal{F}}+C_{\mathcal{G}} \subseteq \mathbb{R}^{n+1}/ \RR e_E$ has dimension $n-1$.
\end{lemma}

\begin{proof}
We will prove the statement by induction.
Suppose $n=1$, so the ground set has size 2. Since the only connected matroid in two elements is the uniform matroid $U_{1,2}$, the only flag, in both $M$ and $M^\perp$, is $\emptyset \subsetneq E$. This flag has $C_{\mathcal{F}} = 0$ and hence in this case $C_{\mathcal{F}} + C_{\mathcal{F}'}$ has dimension $0 = n-1$ in $\RR^2/\RR e_E$.

We now suppose that $M$ is a connected matroid of rank $d$ with ground set $E = \{0, \dots, n\}$, with $|E| \geq 3$. By \cite[Proposition 7.69(1)]{WhiteMatroids1986}, either the deletion $M\backslash\{n\} = (M^\perp/\{n\})^\perp$ or the contraction $M/\{n\} = (M^\perp\backslash\{n\})^\perp$ is a connected matroid. Upon exchanging the roles of $M$ and $M^\perp$, we can suppose that we are in the second case.
Since $M$ is connected, $n$ is not a loop of $M$. Therefore, $\rank(M/\{n\}) = \rank(M) - 1 = d$ and its dual satisfies $\rank((M/\{n\})^\perp) = n -\rank(M/\{n\}) = n-d$.

By our hypothesis, we can find flags of flats $\widetilde{\mathcal F} = \{\emptyset \subsetneq \widetilde{F}_1 \subsetneq \cdots \subsetneq \widetilde{F}_{d-1} \subsetneq E\setminus\{n\}\}$ of $M/\{n\}$ and $\widetilde{\mathcal G} = \{\emptyset \subsetneq \widetilde{G}_1 \subsetneq \cdots \subsetneq \widetilde{G}_{n-d-1} \subsetneq E\setminus\{n\}\}$ of $M^\perp\backslash\{n\}$
such that the Minkowski sum of the cones $C_{\widetilde{\mathcal F}}$ and $C_{\widetilde{\mathcal{G}}}$ has dimension $n-2$ in $\mathbb{R}^{n}/\RR e_{E\backslash\{n\}}$.

Consider now the flag $\mathcal{F} \coloneq \{\emptyset \subsetneq \{n\} \eqcolon F_1 \subsetneq \dots \subsetneq F_{d} \subsetneq E\}$, where $F_i \coloneq \widetilde{F}_{i-1} \cup \{n\}$. 
Since $M/\{n\}$ has no loops (because it is connected), $n$ does not have a parallel element in $M$, so $\{n\}$ is a flat of $M$. Therefore, the flats of $M/\{n\}$ are those obtained by removing $\{n\}$ from flats of $M$ containing this element, so $\mathcal{F}$ is a flag of flats of $M$. 
Similarly, the flag $\mathcal{G} \coloneq \{\emptyset \subsetneq G_1 \subseteq \cdots \subseteq G_{n-d-1}\subseteq E\}$, $G_i\coloneq \text{cl}(\widetilde{G}_i)$ is the smallest flat of $M^\perp$ containing $\widetilde{G}_i$, is a flag of flats for $M^\perp$. The inclusions are strict because $\text{cl}(\widetilde{G}_i)$ is either $\widetilde{G}_i$ itself or $\widetilde{G}_i \cup \{n\}$, and in the latter case, $\text{cl}(\widetilde{G}_j) = \widetilde{G}_j \cup \{n\}$ for all $j \geq i$.

By our induction hypothesis, the linear span of the Minkowski sum of the cones
\[
\langle C_{\widetilde{\mathcal{F}}} + C_{\widetilde{\mathcal{G}}}\rangle_\RR = 
\langle
 e_{\widetilde{F}_1},\dots, e_{\widetilde{F}_{d-1}} ,  e_{\widetilde{G}_1}, \dots,  e_{\widetilde{{G}}_{n-d-1}}
\rangle + e_{E\setminus\{n\}} \subseteq \mathbb{R}^{n}/\RR e_{E\setminus\{n\}}
\]
has dimension $n-2$. Therefore,
\[
 V_{n-1} \coloneq \langle e_{E\setminus\{n\}},
 e_{\widetilde{F}_1},\dots, e_{\widetilde{F}_{d-1}} ,  e_{\widetilde{G}_1},\dots,  e_{\widetilde{G}_{n-d-1}}
\rangle_\RR \subseteq \mathbb{R}^{n}
\]
has dimension $n-1$. Let $W_{n-1} \coloneq V_{n-1}\oplus 0 \subseteq \RR^n \oplus \RR e_{n} = \RR^E$. Since $ e_{n} \notin W_{n-1}$, the vector space
\begin{align*}
V_n \coloneq& \, \langle 
 e_{E\setminus \{n\}},  e_{n},  e_{\widetilde{F}_1},\dots, e_{\widetilde{F}_{d-1}} ,  e_{\widetilde{G}_1}, \dots, e_{\widetilde{G}_{n-d-1}} \rangle_\RR \\
 =& \, \langle 
 e_{E},  e_{F_1},\dots, e_{F_{d}} ,  e_{G_1}, \dots, e_{G_{n-d-1}} \rangle_\RR \subset \RR^{n+1} 
\end{align*}
has dimension $\dim V_n = 1+\dim W_{n-1} = n$. 

Therefore, the vector space spanned by the Minkowski sum of the cones
\[
\langle C_{\mathcal{F}} + C_{\mathcal{G}} \rangle =   \langle 
 e_{E},  e_{F_1},\cdots, e_{F_{d}} ,  e_{G_1}, \cdots, e_{G_{n-d-1}} 
 \rangle +    e_{E}
\]
has dimension $n-1$ in $\mathbb{R}^{n+1}/\RR     e_{E}$, and so $\dim  (C_{\mathcal{F}} + C_{\mathcal{G}}) = n-1$.
\end{proof}

We now finish the proof of \Cref{con:dualdef}.

\begin{proof}[Proof of \Cref{thm: defectiveness}]
We already saw one implication in \Cref{rmk: easydirectionofdualdef}.
   By \Cref{lem: connected}, if $M$ is connected then $\dim (\Sigma_M + \Sigma_{M^\perp}) = n-1$ and so in this case, by \eqref{eq: tropDual},  $\dim (L^{-1})^\vee = n-1$, so $(L^{-1})^\vee$ is not dual-defective.
\end{proof}

\begin{remark}
We have been informed that Dario Antolini has obtained a different, independent proof of \Cref{thm: defectiveness} in his doctoral thesis \cite{Antolini_2026}.
\end{remark}


A \textit{weighted fan} is a fan $\Sigma$ where we associate to each maximal cone $\sigma$ of $\Sigma$ a positive integer weight $\weight_\Sigma(\sigma)$.
For fans coming from projective varieties, the weights  are dictated by the geometry of the variety. For tropicalized linear spaces, or more generally for tropical linear spaces, the correct way to turn
its coarse fan into a weighted fan is to assign to each maximal cone the weight 1. For more details on how to assign weights to tropicalizations we refer to \cite[Section 3.4]{Maclagan2015}

Recall that for every positive integer $m$, $\trop(L^m) = \trop(L) = \Sigma_M$ as sets. However, in general, $\trop(L^m) \neq \trop(L)$ as weighted fans, instead
\begin{equation} \label{eq:weightOfLk}
\weight_{\trop(L^m)}(\sigma) = m^{d-c(M)+1}\weight_{\trop(L)}(\sigma)
\end{equation}
for all maximal cones $\sigma \in \trop(L)$, as shown in \cite[Lemma 4.9]{MatsubaraHeoTelen2026}. This observation is key in the proof of the following proposition.

\begin{proposition}\label{prop: redStep1}
Suppose $M$ is a connected matroid, then $\deg ((L^{-1})^\vee) = 2^{d} \deg(\widetilde{\phi})$, where
\[
\widetilde{\phi}\colon L\times L^{\perp} \dashrightarrow H_+ \subseteq  \mathbb{P}^n, \qquad (x, y) \mapsto xy.
\]
\end{proposition}

\begin{proof}
Since $M$ is connected, by \Cref{thm: defectiveness}, $(L^{-1})^\vee$ is a hypersurface. By \eqref{eq: tropDual} and \cite[Proposition 4.6]{MatsubaraHeoTelen2026}, set-theoretically 
\[
\trop((L^{-1})^\vee) = \trop(H_+).
\]
We now determine the weights of the cones of $ \trop((L^{-1})^\vee)$. Let $\sigma \in \trop((L^{-1})^\vee)$ be a maximal cone. We apply \cite[Theorem 3.12]{sturmfels2007tropical} to the map $q: L^2\times L^\perp \dashrightarrow \PP^n, (x, y) \mapsto x y$,
\begin{align*}
\weight_{\trop((L^{-1})^\vee)}(\sigma)
&= \frac{1}{\deg(q)} \sum_{\tau+\tau^* \supset \sigma} \weight_{\trop(L^2\times L^\perp)}(\tau + \tau^*)[N_\sigma:N_{\tau+\tau^*}]\\
&= \frac{1}{\deg(q)} \sum_{\tau+\tau^* \supset \sigma} \weight_{\trop(L^2)}(\tau)\weight_{\trop(L^\perp)}(\tau^*)[N_\sigma:N_\tau+N_{\tau^*}],
\end{align*}
where the sums run over all maximal cones $\tau \in \trop(L^2)$, $\tau^* \in \trop(L^\perp)$ such that $\tau + \tau^* \supset \sigma$, and $[N_\sigma:N_\tau+N_{\tau^*}]$ is the lattice index of $N_\tau + N_{\tau^*}$ inside $N_\sigma$. Here if $C$ is a cone, $N_C$ is the sublattice of $\ZZ^{n+1}/\ZZ  e_E$ generated by the lattice points of $C$.

As a consequence of projective duality, $\deg(q)= 1$, see \cite[Lemma 4.10]{MatsubaraHeoTelen2026}. By \eqref{eq:weightOfLk}, $\weight_{\trop(L^2)}(\tau) = 2^d \weight_{\trop(L)}(\tau)$ for all maximal cones $\tau \in \trop(L)$.
Therefore,
\begin{align*}
\weight_{\trop((L^{-1})^\vee)}(\sigma) &= 2^d \sum_{\tau+\tau^* \supset \sigma} \weight_{\trop(L)}(\tau)\weight_{\trop(L^\perp)}(\tau^*)[N_\sigma:N_\tau+N_{\tau^*}] \\
&= 2^d \deg(\widetilde{\phi}) \frac{1}{\deg(\widetilde{\phi}) }\sum_{\tau+\tau^* \supset \sigma} \weight_{\trop(L)}(\tau)\weight_{\trop(L^\perp)}(\tau^*)[N_\sigma:N_\tau+N_{\tau^*}] \\
&= 2^d \deg(\widetilde{\phi}) \weight_{\trop(L\star L^\perp)}(\sigma) \\
&= 2^d \deg(\widetilde{\phi}),
\end{align*}
where the next-to-last equality follows again from \cite[Theorem 3.12]{sturmfels2007tropical} and the last equality from the fact that $\trop(L\star L^\perp)=\trop(H_+)$ is a tropicalized linear space and so the weight of every maximal cone is 1.

Finally, \cite[Corollary 3.6.16]{Maclagan2015} tells us how to read off the degree of a variety from the tropicalization:
\[
 \deg((L^{-1})^\vee) = 2^d\deg(\widetilde{\phi})\deg(H_+) = 2^d\deg(\widetilde{\phi}). \qedhere
\]
\end{proof}

\subsection{The bipermutohedral fan and dual degrees}

In this section we recall some of the results from \cite{Ardila2022} which will be useful in our setting to compute the polar degrees of reciprocal linear spaces. We start out by motivating how the bipermutohedral variety and the conormal fan of a matroid, show up in our context. In this way we develop the mechanism that will be used in the proof of \Cref{thm:intro:polardegrees} to translate the intersection-theoretic problem of computing degrees to a combinatorial problem in terms of invariants of matroids.
We close this section by applying this mechanism to the special case of the dual variety to $L^{-1}$, see \Cref{thm: dualdegrees}. Here the corresponding combinatorial problem has already been solved in \cite{Ardila2022} and hence we obtain a closed formula for the dual degree of $L^{-1}$ for any linear space $L\subseteq \PP^n$. In particular this answers \Cref{con:dualdeg} affirmatively.

Let $E=\{0,\dots,n\}$ and let $L\subseteq \PP^n$ be a linear space. Recall the map \begin{align*}
        \phi\colon\PP^n \times (\PP^n)^* \dashrightarrow\PP^n \times (\PP^n)^*, \qquad (x,y)\mapsto (x^{-1},xy).
\end{align*}
In \Cref{sec:swimming} we will relate the polar degrees of $L^{-1}$ to the bidegrees of $\overline{\phi(L\times L^\perp)}\subseteq \PP^n \times (\PP^n)^*$, in the same way \Cref{prop: redStep1} relates the dual degree to the degree of $\widetilde{\phi}$. The ultimate goal of this section is to explain how to compute these bidegrees. The main issue here is that the map $\phi$ is only a rational map and the bidegrees of $\overline{\phi(L\times L^\perp)}$ are governed by how $L\times L^\perp$ intersects the indeterminacy locus of $\phi$. 

Before completely resolving the base locus, we first apply a sequence of blowups in both factors of $\PP^n \times (\PP^n)^*$ individually with the goal of replacing every coordinate subspace in a single factor by a divisor. This construction is a well-known tool in matroid intersection theory and leads to the \emph{permutohedral variety} $X(\Sigma_E)$. This smooth toric variety associated to the permutohedron resolves the Cremona map:
\[
\begin{tikzcd}
& X(\Sigma_E)\arrow[ld,"\alpha"'] \arrow[rd,"\overline{\text{inv}}"] &                    \\ \PP^n \arrow[rr, "\text{crem}", dashed] &                              & \PP^n.
\end{tikzcd}
\]
We refer for example to \cite{BEST2023} for details on the construction of $X(\Sigma_E)$. We get the following diagram. 
\begin{equation}
\label{eq:naiveblowupdiagram}
\begin{tikzcd}
& X(\Sigma_E) \times X(\Sigma_E) \arrow[ld,"\alpha_1\times \alpha_2"] \arrow[rd, dashed,"\overline{\text{inv}}\times \overline{\text{mult}}"] &                    \\ \PP^n \times (\PP^n)^* \arrow[rr, "\phi", dashed] &                              & \PP^n \times (\PP^n)^*.
\end{tikzcd}
\end{equation}
Here the map $\alpha_1\times \alpha_2$ is the birational map obtained by composing all of the blowup maps in the construction of $X(\Sigma_E)$ as an iterated blowup. The map $\overline{\text{inv}}\times \overline{\text{mult}}$ is the composition of $\alpha_1\times \alpha_2$ and $\phi$. 
Notice that this construction only resolves the map $\overline{\text{inv}}$, but not the map $\overline{\text{mult}}$; $\overline{\text{inv}}\times \overline{\text{mult}}$ is still only a rational map.

To resolve the base locus of $\overline{\text{inv}}\times \overline{\text{mult}}$, we replace $X(\Sigma_E)\times X(\Sigma_E)$ by the bipermutohedral variety $X(\Sigma_{E,E})$ of \cite{Ardila2022}. We recall the detailed construction of this toric variety in the later parts of this section. The strict transform of $L\times L^\perp$ gives a subvariety of $X(\Sigma_{E,E})$ which we denote $X_{L,L^\perp}$. We obtain the following extension of diagram \eqref{eq:naiveblowupdiagram} by the resolution of the map $\overline{\text{inv}}\times \overline{\text{mult}}$:
\begin{equation}
\label{eq:cool_diagram}
\begin{tikzcd}[column sep=0pt, row sep=large]
{X_{L,L^\perp}} \arrow[rr, "\iota_L", hook] \arrow[dd] &                                                                                                                                                                                                            & {X(\Sigma_{E,E})} \arrow[ld, "\pi \times \overline \pi"'] \arrow[rd, "\mathrm{inv}\times \mathrm{mult}"] &                                                    & {}  \\
& X(\Sigma_E) \times X(\Sigma_E) \arrow[d, "\alpha_1 \times \alpha_2"'] \arrow[rr, "\overline{\mathrm{inv}} \times \overline{\mathrm{mult}}", dashed] \arrow[rrd, dashed, "\overline{\text{inv}}\times \overline{\mathrm{mult}}"] &                                                                                                          & \PP^n \times (\PP^n)^* \arrow[d, no head, equals]  & {}                                           \\
{}^{\phantom{\perp}} L \times L^\perp \arrow[r, "\subseteq", phantom]               & \phantom{()^*}\PP^n \times (\PP^n)^* \arrow[rr, "\phi", dashed]                                                                                                                              &                                                                                                          & \PP^n \times (\PP^n)^*\makebox[0pt]{.}  &                                             
\end{tikzcd}
\end{equation}
The map $\pi\times \overline{\pi}$ is a birational toric map, defined in \cite[Proposition 2.11]{Ardila2022}. Indeed, on the level of fans, the fan $\Sigma_{E,E}$ is obtained from $\Sigma_E\times \Sigma_E$ by subdividing some two-dimensional cones and choosing a suitable fan structure. In this way $\pi\times \overline{\pi}$ is induced by the identity map on fans. In \Cref{sec:triathlon} we will use this diagram and the push-pull-formula to translate the computation of the bidegrees of $\overline{\phi(L\times L^\perp)}$ into an intersection problem on $X(\Sigma_{E,E})$.

Now we give a short overview on the fan $\Sigma_{E,E}$, the Chow ring $\ChowGrp^\bullet(X(\Sigma_{E,E}))$ and some classes in it, which are relevant for our proof. For a much more detailed exposition we refer the interested reader to \cite{Ardila2022}.

Let $N_{E,E}\coloneq N_E\oplus N_E\coloneq \RR^E/\RR  e_E\oplus \RR^E/\RR f_E $ where $ e_E$ and $f_E$ denote the all-ones vector in the first and second factor respectively. The bipermutohedral fan $\Sigma_{E,E}$ is a complete fan in $N_{E,E}$. We will now describe its rays and cones. 
\begin{itemize}
\item \underline{Rays:} Let $S,T\subseteq E$ be non-empty such that $S\cup T=E$ and $S\cap T\neq E$. We write $S|T$ for the ordered pair $(S,T)$ and call it a \emph{bisubset of $E$}.  Every bisubset of $E$ defines a ray of $\Sigma_{E,E}$ with primitive ray generator
\[
     e_{S|T} \coloneq e_S+f_T=\sum_{i\in S} e_i+\sum_{j\in T}f_j, 
\]
where $ e_i$ and $f_j$ denote the class of the standard basis vectors of $\RR^E$ in the first and second coordinate of $N_{E,E}$, respectively.
\item \underline{Cones:} A collection of rays indexed by bisubsets $S_1|T_1,\dots S_k|T_k$ forms a cone in the bipermutohedral fan $\Sigma_{E,E}$ if and only if, up to reordering the indices, the following three conditions are satisfied:
\[
S_1\subseteq S_2\subseteq \cdots \subseteq S_k, \qquad T_1\supseteq T_2\supseteq \cdots \supseteq T_k, \qquad \bigcup_{i=1}^kS_i\cap T_i \neq E.
\]
\end{itemize}
The variety $X(\Sigma_{E,E})$ in \eqref{eq:cool_diagram} is by construction the toric variety associated with the fan $\Sigma_{E,E}$. The name bipermutohedral fan comes from the fact that the two induced subfans of the subsets of rays $\set{e_{S\vert E}|\emptyset \subsetneq S \subsetneq E}$ and   $\set{e_{E\vert T}|\emptyset \subsetneq T \subsetneq E}$ are two copies of the permutohedral fan $\Sigma_E$. In particular, the two projections $N_{E,E}\to N_E$ induce the map $\pi\times \overline{\pi}\colon X(\Sigma_{E,E})\to X(\Sigma_E)\times X(\Sigma_E)$.  
\begin{lemma}
    The a priori rational maps $\pi\times \overline{\pi}$ and $\textup{inv}\times \textup{mult}$ in diagram \eqref{eq:cool_diagram} are in fact regular.
\end{lemma}
\begin{proof}
    The first factor $\text{inv}$ already has no base locus on the level of products of permutohedra, i.e.\ the map $\overline{\text{inv}}$ in diagram \eqref{eq:cool_diagram} is regular. The second factor and the two maps $\pi,\overline{\pi}$ are regular by \cite[Proposition 2.11]{Ardila2022}.
\end{proof}

Below we give diagram \eqref{eq:cool_diagram} again, but this time on the level of fans instead of toric varieties. In this way all maps become linear. In analogy to the version for varieties we use dashed arrows to indicate that a linear map is not a morphism of fans, i.e.\ does not map cones into cones. We write $\Gamma$ for the fan in $N_E$ that is the normal fan of the standard simplex $\text{Conv}(\set{ e_i | i\in E})$
\begin{equation}
\label{eq:cool_diagram_fans}
\begin{tikzcd}
& {\Sigma_{E,E}} \arrow[ld, "{(x,y)}"'] \arrow[rd, "{(-x,x+y)}"] &                                                    \\
\Sigma_E \times \Sigma_E \arrow[d, "{(x,y)}"] \arrow[rr, "{(-x,x+y)}", dashed] \arrow[rrd, "{(-x,x+y)}", dashed] &                                                                & \Gamma\times \Gamma \arrow[d, no head, equals] \\
\Gamma \times \Gamma \arrow[rr, "{(-x,x+y)}", dashed]                                              &                                                                & \Gamma \times \Gamma \makebox[0pt]{.}  
\end{tikzcd}
\end{equation}
Here the labels on the maps indicate where this map sends a point $(x,y)\in N_E\oplus N_E$. 
Compared to diagram \eqref{eq:cool_diagram}, the above diagram for fans is missing the product of linear spaces $L\times L^\perp$ and its strict transform in $X(\Sigma_{E,E})$. The reason is that these subvarieties are not invariant under the torus action and the inclusions $\iota_L$ and $L\times L^\perp \hookrightarrow \PP^n \times (\PP^n)^*$ are not toric. To place $L\times L^\perp$ into \eqref{eq:cool_diagram_fans} we will replace it by the \emph{conormal fan of $M$} defined in \cite[Section 3.4]{Ardila2022}. This fan is a subfan of $\Sigma_{E,E}$ with the property that its class in the Chow ring of $\Sigma_{E,E}$ agrees with the class of the strict transform $X_{L,L^\perp}$. In particular, for all intersection-theoretic purposes, we can replace $[X_{L,L^\perp}]$ by $[\Sigma_{M,M^\perp}]$. We will make this precise now.

\begin{definition}
\label{def: conormal fan}
Let $M$ be a loopless and coloopless matroid on $E$. The \emph{conormal fan of $M$} is the induced subfan $\Sigma_{M,M^\perp}$ of $\Sigma_{E,E}$ whose support is
\begin{align*}
|\Sigma_{M,M^\perp}|=\trop(M)\times \trop(M^\perp).
\end{align*}
The rays of $\Sigma_{M,M^\perp}$ are labeled by bisubsets $S|T$ such that $S$ is a flat of $M$ and $T$ is a flat of $M^\perp$, we say that $S|T$ is a \emph{biflat}.
\end{definition}

Next we describe the Chow rings of all appearing toric varieties. We generally switch between three equivalent descriptions of these rings: piecewise polynomial functions, torus orbit closures and Minkowski weights. For more details we refer to \cite[Chapter 12.5]{CLS} or the survey paper \cite{ArdilaMantilla2024}.

\begin{itemize}
    \item \textbf{Piecewise linear functions.} For a complete rational simplicial fan $\Sigma\subseteq N_{E,E}$ the Chow ring of the associated smooth complete toric variety $\ChowGrp^\bullet(X(\Sigma))$ (with rational coefficients) is a quotient of the subring of rational piecewise polynomial functions on $\Sigma$. Here a function $f:N_{E,E}\to \RR$ is said to be rational if it takes rational values at the lattice points of $N_{E,E}$ and it is said to be piecewise polynomial on $\Sigma$ if the restriction of $f$ to any cone of $\Sigma$ is given by a polynomial in $\QQ[z_1-z_0,\dots,z_n-z_0,w_1-w_0,\dots,w_n-w_0]$. We take its quotient by the ideal generated by all globally polynomial functions on $\Sigma$.

    In this description pulling back along a map $X(\Sigma_1)\to X(\Sigma_2)$ corresponds to composing a piecewise polynomial function with the linear function of fans $\Sigma_1\to \Sigma_2$.
    \item \textbf{Torus orbit closures.} Fix again a complete rational simplicial fan $\Sigma\subseteq N_{E,E}$. The Chow ring $\ChowGrp^\bullet(X(\Sigma))$  (with rational coefficients) is the quotient ring \begin{align}
    \label{eq:torusorbitchow}
\ChowGrp^\bullet(X(\Sigma))=\QQ[\,x_\rho \mid \rho \in \Sigma(1)\,]/(I+J).
    \end{align}
    Here $I$ is the linear ideal
    \[
        I=\left\langle \textstyle\sum_{\rho \in \Sigma(1)}l(\rho)x_\rho \,\middle|\, l \in N_{E,E}^*  \right\rangle,
    \]
    where $l\in N_{E,E}^*$ means a rational linear function on $N_{E,E}$ and by $l(\rho)$ we mean the value of $l$ at the first lattice point of the ray $\rho$. The ideal $J$ is the monomial ideal generated by all non-cones:
    \[
        J=\langle x_{\rho_1}\cdots x_{\rho_s} \mid s\in \NN,\,  \pos(\rho_1,\dots,\rho_s) \notin \Sigma(s) \rangle.
    \]
    In this description, the degree map is easy to describe. The degree $2n$ part $\ChowGrp^{2n}(X(\Sigma))$ is a one-dimensional $\QQ$-vector space and the class of the monomial $x_{\rho_1}\cdots x_{\rho_{2n}}$ is nonzero and the same for every maximal cone $\pos(\rho_1,\dots,\rho_{2n})$ of $\Sigma$. The degree map $\int_{X(\Sigma)} \colon \ChowGrp^\bullet(X(\Sigma))\to \QQ$ is the unique $\QQ$-linear map which is zero on classes of degree not equal to $2n$ and takes value one on any monomial $x_{\rho_1}\cdots x_{\rho_{2n}}$, where $\text{cone}(\rho_1,\dots,\rho_{2n})\in \Sigma(2n)$.
\end{itemize}

To translate between these two descriptions note first that, since the Chow ring is generated in degree one, it suffices to explain how to translate between the degree one parts of both descriptions. A piecewise linear function $l=(l_\sigma)_\sigma$ is sent to the linear polynomial $\sum_{\rho \in \Sigma(1)}l(\rho)x_\rho$ where $l(\rho)$ is the linear function $l_\sigma$ for any $\sigma$ containing $\rho$, evaluated at the first lattice point of $\rho$.

\begin{itemize} 
    \item \textbf{Minkowski weights.} Let $\Sigma \subseteq N_{E,E}$ be complete, rational and simplicial. Here we give an alternative description of the Chow groups, i.e.\ the $\QQ$-vector spaces $\ChowGrp^k(X(\Sigma))$, without describing the ring structure. By Poincaré duality the intersection pairing \begin{align*}
    \ChowGrp^k(X(\Sigma))\times \ChowGrp^{2n-k}(X(\Sigma))\to \QQ, \qquad (V_1,V_2)\mapsto \int_{X(\Sigma)} V_1V_2
    \end{align*}
    is non-degenerate and induces an isomorphism $\ChowGrp^k(X(\Sigma))\cong \text{Hom}(\ChowGrp^{2n-k}(X(\Sigma)),\QQ)$. By the description above, $\ChowGrp^{2n-k}(X(\Sigma))$ is generated by monomials corresponding to $2n-k$-dimensional cones in $\Sigma$, hence an element of $\text{Hom}(\ChowGrp^{2n-k}(X(\Sigma)),\QQ)$ can be viewed as a function from the $2n-k$-dimensional cones of $\Sigma$ to $\QQ$. The relations among the monomials in $\ChowGrp^{2n-k}(X(\Sigma))$ restrict the functions that can appear in this way to $2n-k$-dimensional \emph{Minkowski weights}, which we now define. 
    
    A function $f \colon \Sigma(2n-k)\to \QQ$ is a Minkowski weight if for every cone $\sigma \in \Sigma(2n-k-1)$ we have \begin{align*}
        \sum_{\sigma'\succ \sigma}f(\sigma')\rho_{\sigma\setminus \sigma'}\in \text{Span}_\RR(\sigma)
    \end{align*}
    where $\rho_{\sigma\setminus \sigma'}$ is the unique primitive ray generator belonging to $\sigma'$ but not to $\sigma$.

    This viewpoint is useful to associate a class in the Chow ring to a weighted balanced subfan $\Sigma'$ of $\Sigma$ of dimension $2n-k$. Indeed $\Sigma'$  gives a Minkowski weight on $\Sigma(2n-k)$ by assigning weight 0 to all cones that do not belong to $\Sigma'$. We write $[\Sigma']$ for this class.
\end{itemize}

The following technical lemma allows us to replace the strict transform of $L\times L^\perp$ in $X(\Sigma_{E,E})$ by the class $[\Sigma_{M,M^\perp}]\in \ChowGrp^{n+1}(X(\Sigma_{E,E}))$. This last class is defined via the Minkowski weight description above by assigning constant weight 1 to all cones in $\Sigma_{E,E}(n-1)$ that belong to $\Sigma_{M,M^\perp}$ and weight zero to all other cones. This in fact makes $\Sigma_{M,M^\perp}$ a balanced weighted fan by \cite[Proposition 3.7]{Ardila2022}. 
\begin{lemma}
\label{lem:classofstricttransform}
    The class of the strict transform $X_{L,L^\perp}$  of $L\times L^\perp$ in $\ChowGrp^{n+1}(X(\Sigma_{E,E}))$ is $[\Sigma_{M,M^\perp}]$.
\end{lemma}
The proof imitates \cite[Lemma 9.8]{BEST2023} for the bipermutohedral fan. For details on the terminology used in this proof we refer to \cite{Maclagan2015}. 
\begin{proof}
On the dense torus of $\PP^n \times (\PP^n)^*$, the tropicalisation of $X^\circ \coloneq (L\times L^\perp)\cap (\TT^n\times \TT^n)$ is supported in $\trop(M)\times \trop(M^\perp)$. We endow this set with the fan structure given by the conormal fan $\Sigma_{M,M^\perp}$. This refines the Gröbner fan $\Sigma_{gb}$ of $X^\circ$ due to \cite[Exercise 4.7.(7)]{Maclagan2015}. By \cite[Proposition 6.4.17]{Maclagan2015} the closure of $X^\circ$ in $X(\Sigma_{gb})$ is a flat tropical compactification. This property is preserved under refinement of $\Sigma_{gb}$ by \cite[Proposition 6.4.14]{Maclagan2015}, hence the closure of $X^\circ$ inside $X(\Sigma_{M,M^\perp})$ is also a flat tropical compactification. On the other hand, under the inclusion $X(\Sigma_{M,M^\perp})\hookrightarrow X(\Sigma_{E,E})$ induced by the inclusion of fans, the closure of $X^\circ$ inside $X(\Sigma_{M,M^\perp})$ is $X_{L,L^\perp}$, i.e.\ the strict transform of $L\times L^\perp$ in $X(\Sigma_{E,E})$. Since $\Sigma_{E,E}$ is a completion of $\Sigma_{M,M^\perp}$, we conclude by \cite[Theorem 6.7.7]{Maclagan2015}, this finishes the proof.
\end{proof}

We also work with the Chow ring of the fan $\Sigma_{M,M^\perp}$, which is not complete; in fact the dimension of this fan is only $n-1$ inside the $2n$-dimensional ambient space $N_{E,E}$. In this case we simply define $\ChowGrp^\bullet(\Sigma_{M,M^\perp})$ to be the quotient ring in \eqref{eq:torusorbitchow}, which is well-defined also for non-complete fans. Concretely, the variables of this ring are $x_{F|G}$ for $F|G$ a biflat of $M$ (\Cref{def: conormal fan}), and a square-free monomial is nonzero if and only if it corresponds to a biflag. This combinatorially defined Chow ring still admits a degree map $\int_{\Sigma_{M,M^\perp}}$ since its top-degree part $\ChowGrp^{n-1}(\Sigma_{M,M^\perp})$ is one-dimensional, see \cite[Proposition 3.7]{Ardila2022}.
Notice that $\ChowGrp^\bullet(\Sigma_{M,M^\perp})$ is the quotient of $\ChowGrp^\bullet(X(\Sigma_{E,E}))$ by the monomial ideal generated by all monomials corresponding to cones in $\Sigma_{E,E}$ which are not cones in $\Sigma_{M,M^\perp}$. The corresponding quotient map gives us a ring homomorphism
\begin{align}
\nonumber
\iota^* \colon \ChowGrp^\bullet(X(\Sigma_{E,E}))&\to \ChowGrp^\bullet(\Sigma_{M,M^\perp}) \\
x_{\rho_1}\cdots x_{\rho_k}&\mapsto \begin{cases}
   x_{\rho_1}\cdots x_{\rho_k} & \text{if } \pos(\rho_1,\dots,\rho_k) \in \Sigma_{M,M^\perp} \\
    0 & \text{otherwisem}
\end{cases}
\end{align}
which we call the pullback map from $X(\Sigma_{E,E})$ to $\Sigma_{M,M^\perp}$. 

Looking again at diagram \eqref{eq:cool_diagram}, to compute the bidegree of the image of $L\times L^\perp$ under $\phi$, we will need to pull back the two hyperplane classes $h_1 = [H \times (\PP^n)^*]$ and $h_2 = [\PP^n\times H]$ following the maps\[
\begin{tikzcd}[sep=large]
\ChowGrp^1(\PP^n\times (\PP^n)^*) \arrow[r, "(\text{inv} \times \text{mult})^*"] & \ChowGrp^1(X(\Sigma_{E,E})) \arrow[r, "\iota^*"] & \ChowGrp^1(\Sigma_{M,M^\perp}).
\end{tikzcd}
\]
We describe the resulting classes in the following lemma.
\begin{lemma}
\label{lem:gammadelta}
     Let $\widetilde{\gamma}\coloneq \iota^*(\textup{inv}\times \textup{mult})^*[H \times (\PP^n)^*]$. For any choice of $i\in E$ this class is given by \begin{align*}
\widetilde{\gamma}=\sum_{\substack{F|G \textup{ biflat} \\ i \notin F}}x_{F|G}.
    \end{align*}
    Let $\delta\coloneq\iota^*(\textup{inv}\times \textup{mult})^*[\PP^n\times H]$. For any choice of $i\in E$ this class is given by 
    \begin{align*}
\delta=\sum_{\substack{F|G \textup{ biflat} \\ i \in F\cap G}}x_{F|G}.
    \end{align*}
\end{lemma}
\begin{proof}
    Write $z_0,\dots,z_n,w_0,\dots,w_n$  for the standard basis of linear functions on $\RR^E\oplus \RR^E$, so that linear functionals on $N_{E,E}$ are given by linear combinations of these functions such that the coefficients of $z_0,\dots,z_n$ sum to zero and the coefficients of $w_0,\dots,w_n$ sum to zero.
    The piecewise linear maps corresponding to $[H\times (\PP^n)^*]$ and $[\PP^n \times H]$ are $\max_{j\in E}(z_i-z_j)$ and $\max_{j\in E}(w_i-w_j)$,  respectively, We compose these maps with the linear map $(-z,z+w)$ and then evaluate the result on the rays of $\Sigma_{E,E}$. After applying $\iota^*$, we obtain the claimed representations.
\end{proof}

The following proposition is the main result of this section and will be used in the computation of the polar degrees of reciprocal linear spaces.
\begin{proposition}\label{prop:howtocomputedegrees}
Let $L\subseteq \PP^n$ be such that $M$ is loopless and coloopless. The bidegree of $\overline{\phi(L\times L^\perp)}$ can be computed in the Chow ring $\ChowGrp^\bullet(\Sigma_{M,M^\perp})$ using the classes $\widetilde{\gamma}$ and $\delta$. More precisely for any $k=0,\dots, d$ we have
\begin{align*}
    \int_{\PP^n \times (\PP^n)^*}[\overline{\phi(L\times L^\perp)}]\cdot[H\times (\PP^n)^*]^k\cdot[\PP^n\times H]^{n-k-1}=\int_{\Sigma_{M,M^\perp}}\widetilde{\gamma}^k\delta^{n-k-1}.
\end{align*}
\end{proposition}

\begin{proof}
Recall that $X_{L,L^\perp}$ is the strict transform of $L\times L^\perp$ in $X(\Sigma_{E,E})$.
It follows from the push-pull formula  \cite[Proposition 2.5(c)]{Fulton1998} and by noting the birationality of $X_{L,L^\perp}\to \overline{\phi(L\times L^\perp)}$, that
\begin{align*}
&\phantom{{}={}} \int_{\PP^n \times (\PP^n)^*}[\overline{\phi(L\times L^\perp)}]\cdot[H \times (\PP^n)^*]^k\cdot[\PP^n\times H]^{n-k-1}\\
&= \int_{X(\Sigma_{E,E})}[X_{L,L^\perp}]\cdot((\text{inv}\times\text{mult})^*[H \times (\PP^n)^*])^k\cdot ((\text{inv}\times\text{mult})^*[\PP^n\times H])^{n-k-1}.
\end{align*}
 By \Cref{lem:classofstricttransform} we have $[X_{L,L^\perp}]=[\Sigma_{M,M^\perp}]$. As a Minkowski weight, $[\Sigma_{M,M^\perp}]$ assigns one to the maximal cones of $\Sigma_{M,M^\perp}$ in $\Sigma_{E,E}$ and 0 to all other cones in $\Sigma_{E,E}$. In particular intersecting a class $[V]\in \ChowGrp^{n-1}(X(\Sigma_{E,E}))$ with $[X_{L,L^\perp}]$ and taking the degree map in $\ChowGrp^{2n}(X(\Sigma_{E,E}))$ is the same as mapping $[V]$ under the combinatorial pullback map $\iota^*$ and applying the combinatorial degree map $\int_{\Sigma_{M,M^\perp}}$. By \Cref{lem:gammadelta}, the pullbacks of the two hyperplane classes map to $\widetilde{\gamma}$ and $\delta$ respectively, which finishes the proof.
\end{proof}

To finish this section, we give a short preview of what will happen in \Cref{sec:triathlon} by applying this machinery to the case $k=0$, which computes the dual degree of $L^{-1}$.
\begin{theorem}
\label{thm: dualdegrees}
    Let $L\subseteq \PP^n$ be a linear space of dimension $0\leq d<n$ such that the matroid $M=\Matroid(L)$ is connected. Then $\deg\, (L^{-1})^\vee = 2^d \beta(M)$.
\end{theorem}
\begin{proof}
    Since $|E|=n+1\geq 2$ connectedness implies that $M$ is loopless and coloopless.
    By \Cref{prop: redStep1} we have $\deg((L^{-1})^\vee)=2^d\deg(\widetilde{\phi})$ where $\widetilde{\phi}=(\text{pr}_2\circ \phi)|_{L\times L^\perp}$. Since $\deg (\phi|_{L\times L^\perp})=1$, it suffices to show that $\deg(\text{pr}_2|_{\overline{\phi(L\times L^\perp)}})=\beta(M)$. Recall that the image $\text{pr}_2(\overline{\phi(L\times L^\perp)})=L\star L^\perp$ is the hyperplane $H_+$, therefore we have $(\text{pr}_2)_*[\overline{\phi(L\times L^\perp)}]=\deg(\text{pr}_2|_{\overline{\phi(L\times L^\perp)}})\cdot[H]$. By the push-pull formula we get \begin{align*}
\deg(\text{pr}_2|_{\overline{\phi(L\times L^\perp)}})
        &=\int_{\PP^n}(\text{pr}_2)_*[\overline{\phi(L\times L^\perp)}] \cdot [H]^{n-1} \\
        &=\int_{\PP^n \times (\PP^n)^*}[\overline{\phi(L\times L^\perp)}]\cdot [\PP^n\times H]^{n-1}\\
        &=\int_{\Sigma_{M,M^\perp}}\delta^{n-1} \\
        &= \beta(M).
    \end{align*}
    Here the second to last equality is precisely \Cref{prop:howtocomputedegrees} and the last equality is \cite[Proposition 4.9]{Ardila2022}.
\end{proof}

\section{Proof of the main theorem} \label{sec:triathlon}

This section contains all details for the proof of our main theorem and it is divided into three parts. 

The first part relates the polar degrees of $L^{-1}$ to the bidegrees of the class $[\overline{\phi(L\times L^\perp)}]$. This is precisely the analog of \Cref{prop: redStep1} where we did the same already for the dual variety $(L^{-1})^\vee$. We will again use tropical methods to extract a power of 2 from the polar degrees.

The second part focuses on computing the bidegrees of the class $[\overline{\phi(L\times L^\perp)}]$ using \Cref{prop:howtocomputedegrees}. We will work in the Chow ring $\ChowGrp^\bullet(\Sigma_{M,M^\perp})$ of the conormal fan of the matroid $M$ to obtain a combinatorial formula summing over flags of flats in the matroid $M$ and involving  beta invariants of minors of $M$. In the case of the dual $(L^{-1})^\vee$, this was already done by \cite[Proposition 4.9]{Ardila2022}. We note that the arguments in this subsection work for all matroids $M$, irregardless of their representability over $\KK$.

In the final part we rewrite the combinatorial formula obtained in part two in terms of the reduced characteristic polynomial of the matroid $M$ which allows us to arrive at the clean formula presented in \Cref{thm:intro:polardegrees}. In this way we conclude the proof of that theorem.

We will assume in \Cref{sec:swimming,sec:cycling} that $M$ is coloopless (and loopless). The main theorem does not have this assumption, we will do this reduction in the proof of the theorem.

\subsection{A simpler map} \label{sec:swimming}

Recall that $\Con(L^{-1})\subseteq \PP^n\times (\PP^n)^*$ denotes the conormal variety of $L^{-1}$. We write $[\Con(L^{-1})]\in \ChowGrp^\bullet(\PP^n\times (\PP^n)^*)_\QQ$ for its class in the Chow ring of $\PP^n\times (\PP^n)^*$. We identify this Chow ring with $\QQ[h_1,h_2]/\langle h_1^{n+1},h_2^{n+1}\rangle$ where $h_1=[H\times (\PP^n)^*]$ and $h_2=[\PP^n\times H]$ are the two pullbacks of hyperplane classes from the two factors. Our first goal is to establish the following proposition.
\begin{proposition}
\label{prop:gettingridof2}
For any $k=0,\dots,d$ we have
\[
\int_{\PP^n\times (\PP^n)^*}[\Con(L^{-1})] \cdot h_1^kh_2^{n-k-1}=2^{d-k}\int_{\PP^n\times (\PP^n)^*}[\overline{\phi(L\times L^\perp)}] \cdot h_1^kh_2^{n-k-1}.
\]
\end{proposition}

\begin{proof}
Consider the following commutative diagram: \begin{equation}
\label{eq:diagramforsquaring}
  \begin{tikzcd}[row sep=tiny]
      \PP^n\times (\PP^n)^* \arrow[r,dashed,"\phi"] & \PP^n\times (\PP^n)^* \\
      \phantom{.}& & \\
    \PP^n\times (\PP^n)^* \arrow[uu,"f_1"]\arrow[r,dashed, "{(x^{-1},x^2y)}"] & \PP^n\times (\PP^n)^* \arrow[uu,"f_2"]\\
    L\times L^\perp \arrow[r,dashed,"\psi"] \arrow[u,"\subseteq",phantom,sloped] & \Con(L^{-1}). \arrow[u,"\subseteq",phantom,sloped]
  \end{tikzcd}
\end{equation}
Here $f_1=f_2$ are the squaring maps in the first factor, i.e.\ they both send a point $(x,y)$ to $(x^2,y)$. The map $\psi$ is precisely the parametrization of $\Con(L^{-1})$ that we deduced in \Cref{prop: parametrization}.

We will mostly work in the two rightmost spaces of this diagram. For better readability we want to distinguish the hyperplane classes on the upper right $\PP^n\times (\PP^n)^*$ from those in the lower right $\PP^n\times (\PP^n)^*$. For that reason we will now use the names $\tilde{h}_1,\tilde{h}_2$ for the two hyperplane classes in the upper right $\PP^n\times (\PP^n)^*$ and as introduced above stick with $h_1,h_2$ on the lower right $\PP^n\times (\PP^n)^*$. Using this notation we clearly have $f_2^*\tilde{h}_1=2h_1$, $f_2^*\tilde{h}_2=h_2$. Hence by the push-pull formula we get
\begin{align*}
  \int_{\PP^n\times (\PP^n)^*}[\Con(L^{-1})]h_1^kh_2^{n-k-1} &= \int_{\PP^n\times (\PP^n)^*}[\Con(L^{-1})] \cdot \Big(\frac{1}{2}f_2^*\tilde{h}_1\Big)^k(f_2^*\tilde{h}_2)^{n-k-1}\\
  &= \frac{1}{2^k}\int_{\PP^n\times (\PP^n)^*}(f_2)_*[\Con(L^{-1})] \cdot \tilde{h}_1^k \tilde{h}_2^{n-k-1} \\
  &= \frac{\deg(f_2|_{\Con(L^{-1})})}{2^k} \int_{\PP^n\times (\PP^n)^*}[\overline{\phi(f_1(L\times L^\perp))}] \cdot \tilde{h}_1^k \tilde{h}_2^{n-k-1}.
\end{align*}
For the last equality we used that the closure of the image of $f_2|_{\Con(L^{-1})}$ is $\overline{\phi(f_1(L\times L^\perp))}$ by commutativity of the above diagram.

There are two things left to prove:
\begin{itemize}
\item We have $\deg(f_2|_{\Con(L^{-1})})=2^{c(M)-1}$. We defer this computation to \Cref{lem:degreeofsquaringCon}.
\item We can rewrite the integral avoiding the map $f_1$, namely
\[
\int_{\PP^n\times (\PP^n)^*}[\overline{\phi(f_1(L\times L^\perp))}]\tilde{h}_1^k\tilde{h}_2^{n-k-1}=2^{d-c(M)+1}\int_{\PP^n\times (\PP^n)^*} [\overline{\phi(L\times L^\perp)}]\tilde{h}_1^k\tilde{h}_2^{n-k-1}.
\]
\end{itemize} 
We will now deal with this last point using tropical intersection theory. By \cite[Theorem 6.7.9]{Maclagan2015} we have
\begin{align*}
  &\phantom{{}={}} \int_{\PP^n\times (\PP^n)^*}[\overline{\phi(f_1(L\times L^\perp))}]\tilde{h}_1^k\tilde{h}_2^{n-k-1}\\
  &= \mathrm{mult}_0\big( \trop(\phi(f_1(L\times L^\perp)))\cap_{\text{st}}\trop(\tilde{h}_1^k)\cap_{\text{st}}\trop(\tilde{h}_2^{n-k-1}) \big).
\end{align*}
Here $\cap_{\text{st}}$ denotes stable intersection and by $\trop(\tilde{h}_1^k)$ we mean the tropicalization of the points with nonzero coordinates in $V\times \PP^n$ for a generic linear subspace $V$ of codimension $k$ and analogously for $\trop(\tilde{h}_2^{n-k-1})$. The result of this stable intersection is a weighted fan and $\text{mult}_0$ denotes its multiplicity at the origin. 
We can tropicalize the monomial map $\phi$ and use \Cref{lem: tropLinSpaces} to see that set-theoretically we have \begin{align*}
  \trop(\phi(f_1(L\times L^\perp))) &= \trop(\phi)(\trop(L^2\times L^\perp))\\
  &= \trop(\phi)(\trop(L\times L^\perp))\\
  &= \trop(\phi(L\times L^\perp)).
\end{align*}
These equalities only hold set-theoretically, \emph{not} on the level of weighted fans. It remains to track how the weights change along these equalities. To track weights under the finite map $\trop(\phi)$ we use \cite[Theorem 3.12]{sturmfels2007tropical} twice. Indeed, by the same reasoning as in \Cref{prop: redStep1}, using \eqref{eq:weightOfLk} and the fact that $\phi$ is birational both restricted to $L\times L^\perp$ and restricted to $L^2\times L^\perp$,
we obtain for every maximal cone $\sigma$ of $\trop(\phi(f_1(L\times L^\perp)))$
\begin{align*}
\weight_{\trop(\phi(f_1(L\times L^\perp)))}(\sigma)=2^{d-c(M)+1}\weight_{\trop(\phi(L\times L^\perp))}(\sigma).
\end{align*}
Since the weights of all cones are rescaled by the same factor, we can extract this factor from the stable intersection and obtain again by \cite[Theorem 6.7.9]{Maclagan2015}
\begin{align*}
  &\phantom{{}={}} \text{mult}_0 \big(\trop(\phi(f_1(L\times L^\perp)))\cap_{\text{st}}\trop(\tilde{h}_1^k)\cap_{\text{st}}\trop(\tilde{h}_2^{n-k-1})\big)\\
  &= 2^{d-c(M)+1}\text{mult}_0\big(\trop(\phi(L\times L^\perp))\cap_{\text{st}}\trop(\tilde{h}_1^k)\cap_{\text{st}}\trop(\tilde{h}_2^{n-k-1}) \big) \\
  &= 2^{d-c(M)+1}\int_{\PP^n\times (\PP^n)^*}[\overline{\phi(L\times L^\perp)}]\tilde{h}_1^k\tilde{h}_2^{n-k-1}. \qedhere
\end{align*}
\end{proof}

\begin{lemma}\label{lem:degreeofsquaringCon}
In diagram \eqref{eq:diagramforsquaring}, the restriction of the squaring map $f_2$ to $\Con(L^{-1})$ is finite of degree $\deg(f_2|_{\Con(L^{-1})})=2^{c(M)-1}$.
\end{lemma}
\begin{proof}
Since the maps $\phi|_{L^2\times L^\perp}$ and $\psi$ are both birational, we have
\begin{align*}
    \deg(f_2|_{\Con(L^{-1})})=\deg(f_2\circ \psi)=\deg(\phi \circ f_1|_{L\times L^\perp})=\deg(f_1|_{L\times L^\perp}).
\end{align*}
As the map $f_1$ is the product of the squaring map in the first factor and identity in the second factor, we get $\deg(f_1|_{L\times L^\perp})=\deg(s|_L)$, where $s\colon\PP^n\to \PP^n$ is the coordinate-wise squaring map. Let $[H]$ and $[\widetilde{H}]$ denote the hyperplane classes on the domain and target of $s$ respectively. Using once more the push-pull formula, we find

\begin{align*}
\deg(L^2) &=\int_{\PP^n}[s(L)][\widetilde{H}]^d \\
&=\int_{\PP^n}\frac{s_*[L]}{\deg(h|_L)}[\widetilde{H}]^d \\
&=\frac{1}{\deg(s|_L)}\int_{\PP^n}[L](2[H])^d \\
&=\frac{2^d}{\deg(s|_L)}.
\end{align*}
Since by \cite[Theorem 2.6]{Dey2020}, $\deg(L^2) = 2^{d-c(M)+1}$ (this uses that $M$ is coloopless),
we conclude that $\deg(h|_L)=2^{c(M)-1}$ as claimed. 
\end{proof}

\subsection{Multidegree} \label{sec:cycling}

Combining \Cref{prop:gettingridof2} and \Cref{prop:howtocomputedegrees}, we reduced the problem of finding the polar degrees of $L^{-1}$ to computing the intersection numbers \[
\int_{\Sigma_{M,M^\perp}}\widetilde{\gamma}^k\delta^{n-k-1}.
\]
The purpose of this section is to carry out this computation and give a first combinatorial formula. To state the main result of this section more precisely we recall the following definition. 
\begin{definition}\label{def:decreasing flag}
Fix the linear order $ 0 < \dots <n$ on the ground set $E=\{0,\dots,n\}$. A flag $\mathcal{F} = \{F_{0} = \emptyset \subsetneq F_1 \subseteq \dots \subseteq F_{k+1} = E\}$ of length $k$ is \emph{decreasing} if $\min F_{i+1} \notin F_{i}$ for all $i=0,\dots,k$. In particular, this implies that every inclusion in $\mathcal{F}$ is strict.
\end{definition}
See \Cref{fig:LatticeOfMex} for an example of a decreasing $4$-flag in $M_{\textup{ex}}$.

\begin{proposition}\label{intersectionNumberToFlagBeta}
For any $k = 0,\dots, d$ we have
\[
\int_{\Sigma_{M,M^\perp}} \widetilde{\gamma}^k \delta^{n-k-1} = \sum_{\substack{\mathcal{F}\textup{ decr.}\\ \textup{of length } k}} \beta(M[\mathcal{F}]),
\]
where the sum goes over all decreasing chains of flats of length $k$ of the form $\mathcal{F}=\{F_0\coloneq \emptyset \subsetneq F_1\subsetneq \cdots \subsetneq F_{k}\subsetneq E\eqqcolon F_{k+1}\}$ and $\beta(M[\mathcal{F}])\coloneq\prod_{i=1}^{k+1}\beta(M|_{F_{i}}/F_{i-1})$.
\end{proposition}

Recall that by \Cref{lem:gammadelta} for any $i \in E$ we have
\[
\widetilde{\gamma} = \sum_{\substack{F|G \textup{ biflat}\\ i \notin F}} x_{F|G}, \qquad \delta = \sum_{\substack{F|G \textup{ biflat}\\ i \in F \cap G}} x_{F|G}
\]
in the Chow ring of the matroid conormal fan.
To prove \Cref{intersectionNumberToFlagBeta} we will first deduce a square-free representation of any power of $\widetilde{\gamma}$. The result will then follow by applying a formula of \cite{Ardila2022}.

\begin{lemma}
\label{lem:expandinggamma}
Let $0\leq k \leq d$.
\begin{enumerate}
\item For all $k$-biflags $\mathcal{F}|\mathcal{G}$ we have
\[
x_{\mathcal{F}|\mathcal{G}} \cdot \widetilde{\gamma} = \sum_{\substack{F|G \textup{ biflat} \\ \mathcal{F}\cup F|\mathcal{G}\cup G \textup{ biflag} \\F \subseteq F_1,\,G\supseteq G_1\\ \min F_1 \notin F}} x_{\mathcal{F}\cup F|\mathcal{G}\cup G}.
\]
Here $\mathcal{F}\cup F|\mathcal{G}\cup G$ means the $(k+1)$-biflag obtained by inserting $F|G$ into $\mathcal{F}|\mathcal{G}$ in the first position.
\item $\widetilde{\gamma}^k = \displaystyle \sum_{\substack{
    \mathcal{F}|\mathcal{G} \textup{ $k$-biflag}\\
    \mathcal{F} \textup{ decr.}
}} x_{\mathcal{F}|\mathcal{G}}$.
\end{enumerate}
\end{lemma}

\begin{proof}
\begin{enumerate}[wide]
\item Write $\mathcal{F}=\{\emptyset\subsetneq F_1\subseteq \cdots \subseteq F_k\subseteq E\}$ and $\mathcal{G}=\{E\supseteq G_1\supseteq \cdots \supseteq G_k\supsetneq \emptyset\}$. In order for a monomial $x_{F|G}x_{\mathcal{F}|\mathcal{G}}$, where $F|G\notin \mathcal{F}|\mathcal{G}$, to be non-zero, one must be able to insert the biflat $F|G$ into the $k$-biflag $\mathcal{F}|\mathcal{G}$ to form a $(k+1)$-biflag. Use $i \coloneq \min F_1$ in the definition of $\widetilde{\gamma}$, i.e.\ use the representation \[
\widetilde{\gamma} = \sum_{\substack{F|G \textup{ biflat}\\ \min(F_1) \notin F}} x_{F|G}.
\] Now for any term $x_{F|G}$ in this representation, suppose that the monomial $x_{F|G}x_{\mathcal{F}|\mathcal{G}}$ is non-zero. Then since $\min(F_1) \notin F$, but $\min(F_1)\in F_j$ for all $j=1,\dots,k$, $F$ must be the first flat of $F \cup \mathcal{F}$. This gives the necessary conditions $F \subseteq F_1$, $G \supseteq G_1$ and ensures that $F|G$ is inserted in the first position of $\mathcal{F}|\mathcal{G}$ to form $\mathcal{F}\cup F|\mathcal{G}\cup G$.  
\item This follows inductively from (i) by observing that the Stanley--Reisner relations force $F \cup \mathcal{F}|G\cup \mathcal{G}$ to be a biflag and the minimum condition ensures that $\mathcal{F}$ is decreasing. \qedhere
\end{enumerate}
\end{proof}

We conclude this section with the proof of \Cref{intersectionNumberToFlagBeta}.

\begin{proof}[Proof of \Cref{intersectionNumberToFlagBeta}]
By \cite[Lemma 4.15 and Proposition 4.18]{Ardila2022}
\[
\int x_{\mathcal{F}|\mathcal{G}} \cdot \delta^{n-k-1} = \begin{cases}
\beta(M[\mathcal{F}]) & \text{if }\mathcal{G} = \mathcal{F}^\perp, \\
0 & \text{otherwise},
\end{cases}
\]
where $\mathcal{F}^\perp \coloneq \{\emptyset \subsetneq \text{cl}_{M^\perp}(E\setminus F_1)\subseteq \dots \subseteq \text{cl}_{M^\perp}(E\setminus F_k)\subseteq E\}$. 
We conclude by applying this to the expansion from \Cref{lem:expandinggamma}:
\[
\int_{\Sigma_{M,M^\perp}} \widetilde{\gamma}^k \delta^{n-k-1} = \sum_{\mathcal{F}\text{ decr.}} \int_{\Sigma_{M,M^\perp}} x_{\mathcal{F}|\mathcal{F}^\perp}\delta^{n-k-1} = \sum_{\mathcal{F}\text{ decr.}} \beta(M[\mathcal{F}]).\qedhere
\]
\end{proof}

\subsection{A closed formula} \label{sec:running}

To recap, we have shown that
\[
\delta_k(\Con(L^{-1})) = 2^{d-k} \int_{X(\Sigma_{E,E})} \widetilde{\gamma}^k \delta^{n-k-1} = 2^{d-k}\sum_{\mathcal{F}\text{ decr.\ $k$-flag}} \beta(M[\mathcal{F}]).
\]
Let $p_M(t) \coloneq \sum_{\mathcal{F}\text{ decr.\ flag}} \beta(M[\mathcal{F}]) t^{|\mathcal{F}|} \in \ZZ[t]$, where $|\mathcal{F}|$ denotes the length of a flag. It remains to establish the following result.
\begin{lemma}\label{lem: polynomialFormula}
We have the equality
\[
p_M(t) = (-t-1)^{d}\overline{\chi}\Big(\frac{1}{t+1}\Big).
\]
\end{lemma}
\begin{proof}
We proceed by induction on the rank of $M$. The base case is
$\rank(M) = 1$. In this case,
the only decreasing flag has length 0, so $p_M(t) = 1$, which agrees with the right hand side of the statement since the reduced characteristic polynomial is
$\overline{\chi}_M(t) = 1$.

Now let $\rank(M)>1$.
For each summand in $p_M(t)$ corresponding to a flag $\mathcal{F}$ of length $k>0$ we single out the largest proper non-empty flat $F=F_k$ of $\mathcal{F}$. We rewrite the summation defining $p_M(t)$ as a sum over flats $F$ appearing in this way. Notice that by the condition of being decreasing, we always have $0\notin F$. Including the summand $\beta(M)$ for the unique length 0 flag yields the recursion
\begin{align*}
    p_M(t)&=\beta(M)+\sum_{\substack{F \text{ flat} \\ 0\notin F\\F\neq \emptyset}}\sum_{\substack{\mathcal{F}' \text{ decr.} \\ \text{ flag in }M|_F}}\beta(M/F)\beta(M|_F[\mathcal{F}'])t^{|\mathcal{F}'|+1} \\
    &=\beta(M)+t\sum_{\substack{F \text{ flat} \\ 0\notin F\\F\neq \emptyset}}\beta(M/F) \cdot p_{M|_F}(t).
\end{align*}
Notice that every occurrence $p_{M|_F}(t)$ on the right hand side uses a strictly smaller matroid since we know $0\notin F$. Hence by induction we know
\begin{align*}
p_M(t)=\beta(M)+t\sum_{\substack{F \text{ flat} \\ 0\notin F\\F\neq \emptyset}}\beta(M/F)(-t-1)^{\rank_M(F)-1}\overline{\chi}_{M|_F}\Big(\frac{1}{1+t}\Big).
\end{align*}
We substitute $t \leadsto \frac{1}{t}-1$ and multiply by $(-t)^{\rank(M)-1}$ to get
\begin{align*}
&\phantom{{}={}} (-t)^{\rank(M)-1}p_M\Big(\frac{1}{t}-1\Big)\\
&=\beta(M)(-t)^{\rank(M)-1}+(t-1)\sum_{\substack{F \text{ flat} \\ 0\notin F\\F\neq \emptyset}}\beta(M/F)(-t)^{\rank(M)-\rank_M(F)-1}\overline{\chi}_{M|_F}(t) \\
&=\beta(M)(-t)^{\rank(M)-1}+\sum_{\substack{F \text{ flat} \\ 0\notin F\\F\neq \emptyset}}\beta(M/F)(-t)^{\rank(M)-\rank_M(F)-1}\chi_{M|_F}(t) \\
&=\sum_{\substack{F \text{ flat} \\ 0\notin F}}\beta(M/F)(-t)^{\rank(M)-\rank_M(F)-1}\chi_{M|_F}(t).
\end{align*}
Here in the last step we include again the summand for the empty flat using that $\chi_{M|_\emptyset}(t)=1$. We could not include this summand earlier since it is the only summand where the (non-reduced) characteristic polynomial is \emph{not} divisible by $t-1$, hence this summand has no reduced characteristic polynomial.

By applying the same substitutions to the equality that we want to prove, we see that it remains to show
\begin{align*}
    \overline{\chi}_M(t) = \sum_{\substack{F \text{ flat} \\ 0\notin F}}\beta(M/F)(-t)^{\rank(M)-\rank_M(F)-1}\chi_{M|_F}(t) \eqcolon q_M(t),
\end{align*}
which is achieved in the following \Cref{lem: formulaforqM}.
\end{proof}

\begin{lemma}\label{lem: formulaforqM}
In the notation above we have $q_M(t) = \overline{\chi}_M(t)$.
\end{lemma}

\begin{proof}
By \cite[Corollary 7.2.7]{WhiteCombGeo1987} and its evaluation at $t=1$, we have
\begin{align}
\overline{\chi}_M(t) &= \sum_{F\text{ flat, }0\notin F}  \mu(\emptyset,F) t^{\rank(M) - \rank_M(F)-1} \label{eq:whiteRedCharPol}\\
(-1)^{\rank(M) - 1} \beta(M) &= \sum_{H\text{ flat, }0\notin H}  \mu(\emptyset,H). \label{eq:whiteBeta}
\end{align}
Let $G$ be a flat, $0 \notin G$, then inserting the definition of $\chi_{M|_G}$ and \eqref{eq:whiteBeta} for $\beta(M/G)$, we see
\begin{align*}
&\phantom{{}={}} \beta(M/G) (-t)^{\rank(M)-\rank_M(G)-1} \chi_{M|_G}(t) \\
&= t^{\rank(M)-\rank_M(G)-1} \sum_{\substack{H\text{ flat of }M/G \\ 0 \notin H}} \mu_{M/G}(\emptyset,H) \sum_{F \text{ flat of } M|_G} \mu_M(\emptyset,F)  t^{\rank_M(G)-\rank_M(F)} \\
&= \sum_{\substack{F,H\text{ flats of }M \\ F \subseteq G \subseteq H,\,0\notin H}}  t^{\rank(M)-\rank_M(F)-1}\mu_M(\emptyset,F) \mu_M(G,H).
\end{align*}
Summing this expression over all flats $G \not\ni 0$, we obtain (notice that $0 \notin H$ and $G \subseteq H$ already implies $0 \notin G$)
\begin{align*}
&\phantom{{}={}} \sum_{\substack{G\text{ flat}\\ 0 \notin G}} \beta(M/G) (-t)^{\rank(M)-\rank_M(G)-1} \chi_{M|_G}(t) \\
&= \sum_{\substack{F,G,H \text{ flats}\\F \subseteq G \subseteq H,\, 0\notin H}} t^{\rank(M)-\rank_M(F)-1} \mu(\emptyset,F) \mu(G,H) \\
&= \sum_{\substack{F,H\text{ flats}\\ F \subseteq H,\, 0 \notin H}}  \mu(\emptyset,F) t^{\rank(M)-\rank_M(F)-1} \sum_{\substack{G\text{ flat}\\ F \subseteq G \subseteq H}} \mu(G,H).
\end{align*}
The rightmost sum over $G$ is zero unless $F=H$, in which case it is $1$ \cite[Proposition 7.1.2]{WhiteCombGeo1987}. This means that we are left with a single sum over $F$ not containing $0$, thus exactly the expression \eqref{eq:whiteRedCharPol} for the reduced characteristic polynomial of $M$.
\end{proof}

We can now prove our main result, \Cref{thm:intro:polardegrees}. Notice that, contrary to the standing assumption in \Cref{sec:swimming,sec:cycling}, $M=\Matroid(L)$ may have coloops.

\begin{proof}[Proof of \Cref{thm:intro:polardegrees}]
We start by reducing to the coloopless case.

Suppose that $M$ has a coloop, say $\{0\}$, and suppose that \Cref{thm:intro:polardegrees} holds for $L' \coloneq L \cap \PP(\KK^{E\setminus\{0\}}) \subseteq \PP^{n-1}$ which satisfies $M' \coloneq \Matroid(L') = \Matroid(L)/\{0\}$. $L^{-1}$ is then a projective cone over $L'^{-1}$ with vertex $[1:0:\dots:0]$, or, in the notation of \Cref{lem:convolution_of_polardegs},
\[
L^{-1} = J(\PP^0,L'^{-1}) \subseteq \PP(\KK^{1+n}).
\]
Applying said lemma, $\mu_i(L^{-1}) = \mu_i(L'^{-1})\cdot \mu_0(\PP^0) = \mu_i(L'^{-1})$ for $i=0,\dots,d-1$ and $\mu_d(L^{-1}) = 0$.
Since $0$ is a coloop, $\rank(M') = \rank(M)-1$, and $\overline{\chi}_M(t) = (t-1)\overline{\chi}_{M'}(t)$.
Therefore,
\begin{align*}
(-2t-1)^d\overline{\chi}_{M}\big(\frac{2t}{2t+1}\big) &= (-2t-1)^{\rank(M')}\big(\frac{2t}{2t+1}-1\big)\overline{\chi}_{M'}\big(\frac{2t}{2t+1}\big) \\
&= (-2t-1)^{\rank(M')-1}\overline{\chi}_{M'}\big(\frac{2t}{2t+1}\big) \\
&= \sum_{j = 0}^{d-1} \mu_j(L'^{-1})t^j\\
&= \sum_{j = 0}^d \mu_j(L^{-1})t^j.
\end{align*}

All that is left to do is to prove the \Cref{thm:intro:polardegrees} in the coloopless case.
By \Cref{prop:gettingridof2}, \Cref{intersectionNumberToFlagBeta} and \Cref{lem: polynomialFormula},
\[
\sum_{i \geq 0} \frac{\delta_i(L^{-1})}{2^{d-i}}t^i =
(-t-1)^{\rank M - 1}\overline{\chi}\left(\frac{1}{t+1}\right).
\]
We obtain the desired expression by substituting $t$ by $t/2$. The polar degrees follow from the equality $\mu_j = \delta_{d-j}$, $j = 0,\dots, d$.
\end{proof}


\begin{remark}
In \cite{Fink2020}, the authors give a Gröbner basis for the ideal of the graph of the Cremona map $\Gamma_L \coloneq \overline{\set{(x,x^{-1}) | x \in L \cap \TT^n}} \subseteq \PP^n \times \PP^n$. This yields a degeneration of $I(\Gamma_L)$ to the Stanley--Reisner ideal of the \emph{extended no broken circuit complex} $\Delta^{\text{ENBC}}_{M,\prec}$ of the matroid $M$. This can be seen as an combinatorial commutative algebra explanation of the fact that the bigraded Hilbert series of $I(\Gamma_L)$, and hence the bidegrees of $\Gamma_L$, can be expressed in terms of the reduced characteristic polynomial of $M$. It would be interesting to see whether the ideal of $\overline{\phi(L\times L^\perp)}$ also admits a particularly nice Gröbner basis which leads to a more algebraic explanation of the results in \Cref{sec:cycling,sec:running}.
\end{remark}

\section{Consequences and Generalizations}
\label{sec:applications}

\subsection{Polar, distance, and Chern degrees}\label{sec:degreeexamples}

We begin this section with examples to illustrate \Cref{thm:intro:polardegrees}.

\begin{example}
Let $L_{\text{ex}}\subseteq \PP^6$ be the 3-dimensional linear subspace defined in \Cref{ex: matroidex} and let $M_{\text{ex}} = \Matroid(L_{\text{ex}})$ be its matroid. The reciprocal linear space $L_{\text{ex}}^{-1}\subseteq \PP^6$ is a projective variety of dimension 3 and degree $10=\mu^+(M_{\text{ex}})$. Its conormal variety is a 5-dimensional subvariety of $\PP^6\times (\PP^6)^*$. The bidegree of this variety is given by
\[
[\Con(L_{\text{ex}}^{-1})]=\mu_0h_1^3h_2^4+\mu_1h_1^4h_2^3+\mu_2h_1^5h_2^2+\mu_3h_1^6h_2 \in \ChowGrp_5(\PP^6 \times (\PP^6)^*).
\]
We determine the coefficients $\mu_0,\dots,\mu_3$ via \Cref{thm:intro:polardegrees}. Indeed, the theorem gives us the following identity of polynomials in one unknown $t$:
\begin{align*}
\mu_0 + \mu_1t + \mu_2t^2 + \mu_3t^3 &= (-2t-1)^3\overline{\chi}_{M_{\text{ex}}}\left(\frac{2t}{2t+1}\right) \\
&= 10 + 34t + 40t^2 + 16t^3.
\end{align*}
Here, the last equality comes from inserting the reduced characteristic polynomial of $M_{\text{ex}}$ that we computed in \Cref{ex: matroidexcontinued}. In particular the degree of $(L^{-1})^\vee$ is 
\[
\deg\,(L^{-1})^\vee=16=2^3\cdot 2=2^d\beta(M_{\text{ex}}).\qedhere
\]
\end{example}

\begin{example}[Uniform matroids]
\label{ex:uniform}
A general linear subspace $L\subseteq \PP^n$ has matroid $M = U_{d+1,n+1}$. Its reduced characteristic polynomial has the particularly simple form
\[
\overline{\chi}_{U_{d+1,n+1}}(t) = \sum_{i=0}^d (-1)^{d-i}\binom{n}{d-i} t^i.
\]
Applying \Cref{thm:intro:polardegrees}, substituting $t\mapsto 2t$, and clearing denominators, we obtain
\begin{align*}
\sum_{i=0}^d \delta_i(L^{-1})\frac{t^i}{2^{d-i}} &= (-t-1)^d \overline{\chi}_M\left(\frac{1}{t+1}\right) \\
&= (-1)^d\sum_{i=0}^d (-1)^{d-i} \binom{n}{d-i} (t+1)^{d-i} \\
&= \sum_{j=0}^d (-1)^{d-j} \binom{n}{j}(t+1)^j.
\end{align*}
After binomial expansion, the coefficient of $t^k$ in the previous expression is given by
\[
\sum_{j=k}^d (-1)^{d-j} \binom{n}{j} \binom{j}{k} = \sum_{j=k}^d (-1)^{d-j} \binom{n}{k} \binom{n-k}{j-k} = \binom{n}{k} \binom{n-k-1}{d-k}.
\]
Thus,
\[
\delta_k(L^{-1}) = 2^{d-k} \binom{n}{k} \binom{n-k-1}{d-k},  \qquad  \mu_i(L^{-1}) = 2^i \binom{n-d-1+i}{i} \binom{n}{d-i} .
\]
In particular, since the uniform matroid is connected, $\deg\, (L^{-1})^\vee = \delta_0(L^{-1}) = 2^d \binom{n-1}{d}$, in coherence with \cite[Theorem 4.11]{MatsubaraHeoTelen2026}. The formula $\mu_1 = 2(n-d)\binom{n}{d-1}$ also appears in \cite[Example 4.1]{sturmfels2016hurwitzformprojectivevariety}, compare \Cref{cor:mu0mu1formula}(ii) below.
\end{example}

\begin{example}\label{exa:Fano}
The \emph{Fano plane} $\PP^2(\mathbb{F}_2)$ famously gives rise to a matroid $F_7$ of rank $3$ on $7$ elements realizable only over fields of characteristic $2$ as $L_{\text{Fano}} \subseteq \PP^6$. However, we invite the reader to check that our proof of \Cref{thm:intro:polardegrees} never crucially relies on the standing assumption of $\operatorname{char}(\KK) = 0$. We have
\begin{align*}
\overline{\chi}_{F_7}(t) = t^2 - 6t + 8 \qquad \leadsto \qquad \mu_0 = 8,\ \mu_1 = 20,\ \mu_2 = 12,
\end{align*}
which can be confirmed by computing the bidegrees of $\Con(L_{\text{Fano}}^{-1})$ using computer algebra.

Note however that in this case $\mu_2$ is \emph{not} the dual degree: Since squaring in characteristic $2$ is set-theoretically the identity, we have that $(L_{\text{Fano}}^{-1})^\vee = H_+$. In particular, $L_{\text{Fano}}^{-1}$ does not satisfy projective duality, $\Con(L_{\text{Fano}}^{-1}) \to (L_{\text{Fano}}^{-1})^\vee$ is an \emph{inseparable} map.
This is the reason that $\deg\, (L_{\text{Fano}}^{-1})^\vee \neq \mu_2(L_{\text{Fano}}^{-1})$.
The proof of \Cref{prop: redStep1} relies on projective duality and hence also does not give the correct result.
\end{example}

\Cref{thm:intro:polardegrees} gives a formula for all polar degrees in the form of the generating function of $(\mu_i)_i$ resp.\ $(\delta_i)_i$. This yields the promised alternative proof of \Cref{con:dualdef,con:dualdeg}, which we proved in \Cref{thm: defectiveness,thm: dualdegrees}, about defectivity and dual degree.

\begin{proof}[Proof of \Cref{cor:introconjs}]
We use the second formula of \Cref{thm:intro:polardegrees} to get
\begin{align*}
\delta_0(L^{-1})=\bigg(\sum_{i=0}^d\delta_i(L^{-1})t^i\bigg)\bigg|_{t=0}=(-2)^d\overline{\chi}_M\left(\frac{2}{2}\right)=2^d\beta(M).
\end{align*}
A variety is dual defective if and only if $\delta_0(X)=0$, hence $L^{-1}$ is dual defective if and only if $\beta(M) = 0$. Since the groundset of $M$ has size $n+1>1$, this holds if and only if $M$ is disconnected. In the non-defective case, one has $\delta_0(X) = \deg(X^\vee)$, hence $\deg\,(L^{-1})^\vee = 2^d \beta(M)$.
\end{proof}
With slightly more care, we can also give the dual defect and dual degree in the disconnected case.

\begin{corollary}\label{cor:defect}
The codimension of $(L^{-1})^\vee$ in $(\PP^n)^*$ equals the number of connected components $c(M)$. Its degree is
\[
\deg\,(L^{-1})^\vee = 2^{d+1-c(M)}\left.\frac{\chi_M(t)}{(t-1)^{c(M)}}\right|_{t=1} = \prod_{C \textup{ conn.\ comp.\ of }M} 2^{\rank(C)-1} \beta(C).
\]
Furthermore, the polar degrees of $M$ are the convolution of the polar degrees of its connected components.
\end{corollary}
\begin{proof}
This can either be deduced from the connected case (\Cref{cor:introconjs}) together with the convolution identity from \Cref{lem:convolution_of_polardegs}, or by the fact that the characteristic polynomial is multiplicative under direct sums:
\[
\chi_M(t) = \prod_{C} \chi_{C}(t), \qquad \overline{\chi}_M(t) = (t-1)^{c(M)-1} \prod_{C} \overline{\chi}_{C}(t). \qedhere
\]
\end{proof}

\Cref{thm:intro:polardegrees} also gives a new proof of the expression of the first two polar degrees of reciprocal linear spaces in terms of matroid invariants. For any projective variety $X\subseteq \PP^n$, the zeroth polar degree $\mu_0(X)$ is the degree of the variety $X$, while the next polar degree $\mu_1(X)$ is the degree of the Hurwitz form $\operatorname{Hu}_X$ \cite[Theorem 9]{Kohn2021}.

\begin{corollary}\label{cor:mu0mu1formula}
\begin{enumerate}
\item The degree of $L^{-1}$ is $\mu_0(L^{-1}) = \mu^+(M)$.
\item The Hurwitz degree of $L^{-1}$ is $\mu_1(L^{-1}) = 2(-1)^{\rank(M)}(\rank(M)\chi_M(0) + \chi'_M(0))$.
\end{enumerate}
\end{corollary}

The formula for $\deg L^{-1}$ first appeared in \cite{orlik1992arrangements}.
Its Hurwitz degree was first mentioned in \cite[Example 4.1]{sturmfels2016hurwitzformprojectivevariety}, which references \cite[Proof of Proposition 33]{Sanyal2013}.

\begin{proof}
To recover the degree of $L^{-1}$ we evaluate the formula for the generating function of the polar degrees derived in \Cref{thm:intro:polardegrees} at $t=0$:
\[
\mu_0(L^{-1}) = (-1)^d\overline{\chi}_M(0) = \mu^+(M).
\]
To recover the Hurwitz degree, we first differentiate the generating function, obtaining
\[
\sum_{j=1}^dj\cdot \mu_j(L^{-1})t^{j-1} = -2d(-2t-1)^{d-1} \overline{\chi}_M\bigg(\frac{2t}{2t+1}\bigg) + \frac{2(-2t-1)^d}{(2t+1)^2} \overline{\chi}'_M\bigg(\frac{2t}{2t+1}\bigg).
\]
We now evaluate this equality at $t=0$, and we get
\begin{align*}
\mu_1(L^{-1}) &= 2d(-1)^{d}\overline{\chi}_M(0) + 2(-1)^{d}\overline{\chi}'_M(0)\\ 
&= 2d(-1)^{d+1}\chi_M(0) + 2(-1)^d(-\chi_M'(0)-\chi_M(0)) \\
&= 2(-1)^{d+1}(d\chi_M(0)+\chi_M'(0)+\chi_M(0)) \\
&= 2(-1)^{d+1}((d+1)\chi_M(0)+\chi_M'(0)),
\end{align*}
where the second equality uses $\overline{\chi}_M(0) = -\chi_M(0)$ and $\overline{\chi}_M'(0) = -\chi_M'(0) - \chi_M(0)$.
\end{proof}

Next, we discuss an application to Euclidean distance optimization. If $X \subseteq \CC^N$ is an irreducible variety and $Q \in S^2(\CC^N)^*$ a non-degenerate quadric, then the \emph{Euclidean distance degree} $\DD_Q(X)$ is the number of critical points of $x \mapsto Q(x-u)$ on $X_\reg$ for general $u \in \CC^N$ \cite{Draisma2015}. The distance degree of a projective variety $X \subseteq \PP^{N-1}$ is that of its affine cone. If $Q$ is chosen generally, then the distance degree attains its maximum value (among all quadrics); this value is the \emph{generic distance degree} $\DDgen(X)$.

\begin{theorem}
We have $\DDgen(L^{-1}) = (-3)^d \overline{\chi}_M(\frac{2}{3})$. On the other hand, $\DD_{Q}(L^{-1})$ is \emph{not} a matroid invariant for the standard quadric $Q = \sum_{i=0}^n x_i^2$.
\end{theorem}
\begin{proof}
The generic distance degree of a projective variety is the sum of its polar degrees \cite[Theorem 5.4]{Draisma2015}. Thus, evaluating the polar polynomial from \Cref{thm:intro:polardegrees} at $t=1$ yields the generic distance degree.

To show that the specific Euclidean distance degree is not a matroid invariant, we provide a counterexample. Let $L = V(x_0+x_2+x_2) \subseteq \PP^2$, then $L^{-1} = V(x_1x_2+x_0x_2+x_0x_2) \subseteq \PP^2$. $L^{-1}$ intersects the isotropic quadric $V(Q)$ non-transversally in two ordinary double points, hence $\DD_Q(L^{-1}) = 2$, see for example \cite[Corollary 6.3]{Aluffi2018}. On the other hand, $L' = V(x_0+2x_1+x_2)$ has the same matroid $\Matroid(L') = \Matroid(L) = U_{2,3}$, but $L'^{-1}$ is transversal to $V(Q)$, so $\DD_Q(L^{-1}) = 2 \neq 4 = \DD_Q(L'^{-1})$.
\end{proof}

Finally, we give a combinatorial expression for the degrees of the Chern--Mather classes of $L^{-1}$. Recall that if $X$ is a potentially singular projective variety, MacPherson \cite{MacPherson1974} defined the \emph{Chern--Mather classes} $c^{\mathrm{M}}(X) \in \mathrm{A}_\bullet(X)$ via the Nash blow-up $\nu\colon \widetilde{X} \to X$ as $\nu_*(c({\mathcal T}_{\widetilde{X}}) \cap [\widetilde{X}])$ (see also \cite[Example 4.2.9]{Fulton1998}). On smooth varieties, it recovers the usual Chern class $c(X) = c(\mathcal{T}_X) \in \mathrm{A}^\bullet(X)$ as $c^{\mathrm{M}}(X) = c(X)\cap [X]$.

Piene \cite[Theorème 3]{Piene1988} showed that the polar classes $P_k(X) \in \mathrm{A}_{d-k}(X)$ and Chern--Mather classes $c^{\mathrm{M}}_i(X) \in \mathrm{A}_{d-i}(X)$ of a $d$-dimensional variety are related by the (involutive) operation
\[
c^{\mathrm{M}}_{k}(X) = \sum_{i=0}^k (-1)^{d-i}\binom{d+1-k+i}{i} \cdot [H]^i \cap P_{k-i}(X).
\]
The \emph{Chern--Mather degrees} of a projective variety $\iota \colon X \hookrightarrow \PP^n$ are then given by
\[
c_{\mathrm{Ma}}(X) \coloneq  \iota_*c^{\mathrm{M}}(X) = \sum_{i=0}^d \deg_{\PP^n} c^{\mathrm{M}}_{i}(X) \cdot h_1^{n-d+i} \in \ChowGrp^\bullet(\PP^n) = \ZZ[h_1]/\langle h_1^{n+1}\rangle.
\]
Piène's formula allows to express Chern--Mather degrees in terms of polar degrees, this allows for a combinatorial formula of $c_{\mathrm{Ma}}(L^{-1})$:

\begin{theorem}\label{thm:ChernMather}
The degrees of the Chern--Mather classes of $L^{-1}$ can be computed as
\[
c_{\mathrm{Ma}}(L^{-1}) = (1+h_1)h_1^{n-d}(h_1-1)^d\overline{\chi}_M\left(\frac{2h_1}{h_1-1}\right)\in  \ChowGrp^\bullet(\PP^n) = \ZZ[h_1]/\langle h_1^{n+1}\rangle
\]
\end{theorem}
\begin{proof}
Equation 2.1 in \cite{Aluffi2017} states
\begin{equation}\label{eq: aluffi}
c_{\mathrm{Ma}}(X) = (-1)^{n-1+d}(1+h_1)^{n+1}  {\operatorname{pr}_1}_*\left(\frac{1}{1+h_1+h_2} \cdot   [\Con(X)]\right)
\end{equation}
for every projective variety $X\subseteq \PP^n$.
By performing the formal substitution $t = \frac{h_2}{h_1}$ in the combinatorial description from \Cref{thm:intro:polardegrees} we get
\[
[\Con(L^{-1})]=\sum_{i=0}^d\delta_i(L^{-1})h_1^{n-i}h_2^{i+1}=h_1^{n-d}h_2(-2h_1-h_2)^d\overline{\chi}_M\left(\frac{2h_1}{2h_1+h_2}\right).
\]
Write $\overline{\chi}_M(t)=\sum_{i=0}^da_it^i$, then
\[
\frac{1}{1+h_1+h_2}\cdot [\Con(L^{-1})]=(-1)^dh_1^{n-d}\sum_{i=0}^da_i(2h_1)^i\frac{(2h_1+h_2)^{d-i}}{1+h_1+h_2}.
\]
To evaluate equation \eqref{eq: aluffi}, we need to extract the coefficient of $h_2^n$ as a polynomial in $h_1$ in the previous expression. This follows from the formal observation that
\begin{align*}
\operatorname{coeff}_{h_2^n} \frac{(2h_1+h_2)^{d-i}}{1+h_1+h_2} = \frac{(-1)^{n-1}}{(1+h_1)^n}(h_1-1)^{d-i}
\end{align*}
for all $i=0,\dots,d$. By equation \eqref{eq: aluffi} we conclude that \begin{align*}
    c_{\mathrm{Ma}}(X) &= (1+h_1)^{n+1} h_1^{n-d}\sum_{i=0}^da_i\frac{1}{(1+h_1)^n}(2h_1)^i(h_1-1)^{d-i} \\
    &=(1+h_1)h_1^{n-d}(h_1-1)^d\overline{\chi}_M\left(\frac{2h_1}{h_1-1}\right). \qedhere
\end{align*}
\end{proof}

\begin{example}
If $L$ is a general linear space of dimension $1$, meaning $\Matroid(L) = U_{2,n+1}$, then $X = L^{-1}$ is a smooth curve of degree $\mu^+(U_{2,n+1}) = n$, that is, a rational normal curve. In this case,
\[
c^{\mathrm{M}}(X) = [X] + 2[\text{pt}] \in \ChowGrp_\bullet(X), \qquad c_{\mathrm{Ma}}(X) = n\cdot h_1^{n-1} +2\cdot h_1^n \in \ChowGrp^\bullet(\PP^n).
\]
We can verify this basic fact using \Cref{thm:ChernMather}. Indeed, using $\overline{\chi}_M(t)=t-n$ we obtain
\[
c_{\mathrm{Ma}}(L^{-1}) =(1+h_1)h_1^{n-1}(2h_1-n(h_1-1)) =2h_1^n+nh_1^{n-1}. \qedhere
\]

\end{example}

\subsection{Powers of reciprocal linear spaces}

In this final section, we give a generalization of \Cref{thm:intro:polardegrees} to higher powers of $L^{-1}$. Indeed, throughout \Cref{sec:triathlon}, we focused on the case of reciprocal linear spaces, however, all main ideas are still applicable for any negative power of a linear space. More specifically, let $m \geq 1$ and consider the map
\[
\PP^n \dashrightarrow \PP^n, \qquad 
    [x_0:\dots:x_n] \mapsto [x_0^{-m}:\dots:x_n^{-m}].
\]
Let $L^{-m}$ be the closure of the image of $L$ under this map. We determine its polar degrees.

\begin{theorem}\label{thm: polardegswithk}
Let $L\subseteq \PP^n$ be a linear subspace of dimension $0\leq d<n$ and fix $m\geq 1$. Let $\mu_j(L^{-m})$ be the polar degrees of $L^{-m} \subseteq \PP^n$, indexed such that $\mu_0(L^{-m}) = \deg L^{-m}$. Let $M=\Matroid(L)$ be the matroid realized by $L$, then
    \[
    \sum_{j = 0}^d \mu_j(L^{-m})t^j = m^{1-c(M)}(-(m+1)t-m)^{d}\, \overline{\chi}_M\left(\frac{(m+1)t}{(m+1)t+m}\right).
    \]
Equivalently, setting $\delta_i(L^{-m}) \coloneq \mu_{d-i}(L^{-m})$ we have
\[
\sum_{i = 0}^d \delta_i(L^{-m})t^i = m^{1-c(M)}(-mt-m-1)^{d} \, \overline{\chi}_M\left(\frac{m+1}{mt+m+1}\right).
\]
In particular if $M$ is connected, then $\deg\, (L^{-m})^\vee = (m+1)^d\beta(M)$.
\end{theorem}

\begin{remark}
In the above expressions, when $M$ has several connected components, $m$ appears with a negative power. It may thus seem unclear why the above polynomial has integer coefficients. However, if $M$ decomposes into $c(M)$ connected components labeled
$M_1, \dots, M_{c(M)}$, then
\[
\overline{\chi}_M(s) = \frac{\chi_M(s)}{s-1} = \frac{\chi_{M_1}(s)\dotsm \chi_{M_{c(M)}}(s)}{s-1} = (s-1)^{c(M)-1}\overline{\chi}_{M_1}(s) \dotsm \overline{\chi}_{M_{c(M)}}(s)\in \ZZ[s].
\]
Therefore the right hand side of the first equality of \Cref{thm: polardegswithk} rewrites as
    \[
 (-(m+1)t-m)^{d+1-c(M)}\cdot \overline{\chi}_{M_1}\left(\frac{(m+1)t}{(m+1)t+m}\right) \dotsm \overline{\chi}_{M_{c(M)}}\left(\frac{(m+1)t}{(m+1)t+m}\right),
    \]
 which clearly has integer coefficients. Analogously to (the proof of) \Cref{cor:defect}, $L^{-m}$ is dual defective if and only $M$ is disconnected.
\end{remark}

As in the proof of our main result, \Cref{thm:intro:polardegrees}, the strategy is to find a parametrization of the conormal variety $\Con(L^{-m})$ and reduce it to the a computation of the intersection numbers $\int_{\Sigma_{M,M^\perp}}\widetilde{\gamma}^k\delta^{n-k-1}$.
Relative to the proof of \Cref{thm:intro:polardegrees}, to prove \Cref{thm: polardegswithk} we only need to adapt the proofs of \Cref{prop: parametrization} and \Cref{prop:gettingridof2}. The rest of the argument is identical.
As in \Cref{sec:triathlon}, from now on until the proof of \Cref{thm: polardegswithk}, we will  suppose that $M$ is coloopless. 

\begin{proposition}
    \label{prop: paramerizationwithk}
The rational map
\begin{align*}
\psi_m \colon L \times L^\perp \dashrightarrow \PP^n \times (\PP^n)^*, \qquad (x, y) \mapsto  (x^{-m}, x^{m+1}y)
\end{align*}
rationally parametrizes $\Con(L^{-m})$. Its degree is $\deg(\psi_m) = m^{c(M)-1}$.
\end{proposition}
\begin{proof}
To prove that $\psi_m$ parametrizes $\Con(L^{-m})$ we apply the same argument as in \Cref{prop: parametrization} with replacing $-1$ by $-m$ both in the proof of the proposition and in the proof of \cite[Proposition 4.1]{MatsubaraHeoTelen2026} which is cited in the proof. Additionally we need to replace $L^{-1}\cap \TT^n$ by $(L^{-m})_{\reg}\cap\TT^n$ since for $m>1$ the variety $L^{-m}$ can have singularities even on the torus. The arguments still all hold true for the smooth points and by taking the closure we still get all of $\Con(L^{-m})$.

To compute the degree, we note that $\phi\circ \psi_m$ is the $m$-th power map in the first factor and identity in the second factor. This map has degree $m^{c(M)-1}$, which can be seen with the same argument using \cite[Theorem 2.6]{Dey2020} as in the proof of \Cref{lem:degreeofsquaringCon}. Since $\phi$ is clearly birational on the image of $\psi_m$, we conclude.
\end{proof}

We now apply the same strategy as in \Cref{prop:gettingridof2} to get rid of the exponents in the $x$ variables.
\begin{proposition}\label{prop: gettingridofthek}
For $k=0,\dots,d$ we have
\[
\int_{\PP^n \times (\PP^n)^*}[\Con(L^{-m})]h_1^kh_2^{n-k-1}=(m+1)^{d-k}m^{k+1-c(M)}\int_{\PP^n \times (\PP^n)^*}[\overline{\phi(L\times L^\perp)}] h_1^kh_2^{n-k-1}.
\]
\end{proposition}
\begin{proof}
Consider the following commutative diagram:
    \begin{equation}
      \label{eq:diagramfork}
      \begin{tikzcd}[row sep=tiny]
          \PP^n \times (\PP^n)^* \arrow[rr,dashed,"\phi"] & & \PP^n \times (\PP^n)^* \\
          \phantom{.}& & \\
        \PP^n \times (\PP^n)^* \arrow[uu,"g_1"]\arrow[rr,dashed, "{(x^{-m},x^{m+1}y)}"] & & \PP^n \times (\PP^n)^* \arrow[uu,"g_2"]\\
        L\times L^\perp \arrow[rr,dashed,"\psi_m"] \arrow[u,"\subseteq",phantom,sloped] & &\Con(L^{-m}), \arrow[u,"\subseteq",phantom,sloped]
      \end{tikzcd}
  \end{equation}
where $g_1(x,y) = (x^{m(m+1)},y^m)$ and $g_2(x,y) = (x^{m+1},y^m)$.

Let $\tilde{h}_1,\tilde{h}_2$ be the two hyperplane classes in the upper right $\PP^n \times (\PP^n)^*$ and $h_1,h_2$ be the ones on the lower right $\PP^n \times (\PP^n)^*$. Using this notation we clearly have $g_2^*\tilde{h}_1=(m+1)h_1$, $g_2^*\tilde{h}_2=mh_2$. By the push-pull formula we get
\begin{align*}
&\phantom{{}={}} \int_{\PP^n \times (\PP^n)^*}[\Con(L^{-m})]h_1^kh_2^{n-k-1}\\
&= \int_{\PP^n \times (\PP^n)^*}[\Con(L^{-m})]\left(\frac{1}{m+1}g_2^*\tilde{h}_1\right)^k\left(\frac{1}{m}g_2^*\tilde{h}_2\right)^{n-k-1}\\
&=(m+1)^{-k}m^{-n+k+1}\int_{\PP^n \times (\PP^n)^*}(g_2)_*[\Con(L^{-m})]\tilde{h}_1^k\tilde{h}_2^{n-k-1} \\
&=(m+1)^{-k}m^{-n+k+1}\deg(g_2|_{\Con(L^{-m})})\int_{\PP^n \times (\PP^n)^*}[\overline{\phi(g_1(L\times L^\perp))}]\tilde{h}_1^k\tilde{h}_2^{n-k-1}
\end{align*}
We have $\deg(g_2|_{\Con(L^{-m})})=(m(m+1))^{c(M)-1}$ in complete analogy to the proof of \Cref{lem:degreeofsquaringCon}.

By \Cref{lem: tropLinSpaces} we have set-theoretically
\begin{align*}
      \trop(g_1(L\times L^\perp)) =\trop(L\times L^\perp). 
  \end{align*}
  Using the same argument as as in \cite[Lemma 4.9]{MatsubaraHeoTelen2026}, the weight of every maximal cone in $\trop(g_1(L\times L^\perp))$ is 
  \[
(m(m+1))^{d-c(M)+1}m^{n-d-1-c(M^\perp)+1}=(m+1)^{d-c(M)+1}m^{n-2c(M)+1}.
\]
Analogously to the proof of \Cref{prop: redStep1}, combining \cite[Theorem 6.7.9]{Maclagan2015} and \cite[Theorem 3.12]{sturmfels2007tropical}, we conclude 
  \begin{align*}
     &\int_{\PP^n \times (\PP^n)^*}[\overline{\phi(g_1(L\times L^\perp))}]\tilde{h}_1^k\tilde{h}_2^{n-k-1} \\&= (m+1)^{d-c(M)+1}m^{n-2c(M)+1}\int_{\PP^n \times (\PP^n)^*}[\overline{\phi(L\times L^\perp)}]\tilde{h}_1^k\tilde{h}_2^{n-k-1}. \qedhere
  \end{align*}
\end{proof}

We can now use the results derived in \Cref{sec:triathlon} to give a proof of \Cref{thm: polardegswithk}.

\begin{proof}[Proof of \Cref{thm: polardegswithk}]
We can reduce to the coloopless case using the same argument as in the proof of \Cref{thm:intro:polardegrees}.
    Using \Cref{prop: gettingridofthek},  \Cref{intersectionNumberToFlagBeta} and \Cref{lem: formulaforqM}, we have
    \[
    \sum_{i=0}^d \frac{\delta_i(L^{-m})}{(m+1)^{d-i}m^{i-c(M)+1}} t^i =
    (-t-1)^{d}\overline{\chi}_M\bigg(\frac{1}{t+1}\bigg).
    \]
    By substituting $t \leadsto \frac{m}{m+1}t$ and multiplying by $(m+1)^dm^{-c(M)+1}$, we obtain the desired formula.
\end{proof}

\begin{remark}
It would also be interesting to explore the polar degrees of \emph{positive} powers of linear spaces $L^m$ in the spirit of \cite{Dey2020}. The conormal variety is still rationally parametrized by $(x,y) \mapsto (x^m, x^{-m+1}y)$, however, the bipermutahedral variety $X(\Sigma_{E,E})$ does \emph{not} resolve the rational map $(x,y) \mapsto (x, x^{-1}y)$, so our approach does not directly carry over to this situation. The \enquote{correct} fan $(-\Sigma_E \times \Sigma_E) \cap \Sigma_{E,E}$ and its matroid version, the \emph{square conormal fan} $\Sigma_{-M,M^\perp} \coloneq (-\Sigma_M \times \Sigma_{M^\perp}) \cap \Sigma_{E,E}$, have been introduced very recently in \cite{bath2026tropicalresolutionsconfigurationhypersurfaces}. We expect similar formulae akin to \Cref{thm: polardegswithk} to hold in this case, and leave this direction to future research.
\end{remark}

\printbibliography
\end{document}